\documentclass[11pt]{article}
\usepackage{amsmath}
\usepackage{amssymb}
\usepackage{amsfonts}
\usepackage{hyperref}
\usepackage[dvips]{graphicx}

\newtheorem{theorem}{Theorem}

\newtheorem{corollary}[theorem]{Corollary}

\newtheorem{definition}[theorem]{Definition}

\newtheorem{lemma}[theorem]{Lemma}

\newtheorem{proposition}[theorem]{Proposition}

\newenvironment{proof}[1][Proof]{\noindent\textbf{#1.} }{\ \rule{0.5em}{0.5em}}
\renewcommand{\theequation}{\thesection.\arabic{equation}}
\input{tcilatex}
\begin{document}

\title{Deformations of $G_{2}$-structures with torsion}
\author{Sergey Grigorian \\
Simons Center for Geometry and Physics\\
Stony Brook University\\
Stony Brook, NY 11794\\
USA}
\maketitle

\begin{abstract}
We consider non-infinitesimal deformations of $G_{2}$-structures on $7$%
-dimensional manifolds and derive an exact expression for the torsion of the
deformed $G_{2}$-structure. We then specialize to a case when the
deformation is defined by a vector $v$ and we explicitly derive the
expressions for the different torsion components of the new $G_{2}$%
-structure in terms of the old torsion components and derivatives of $v$. In
particular this gives a set of differential equations for the vector $v$
which have to be satisfied for a transition between $G_{2}$-structures with
particular torsions. For some specific torsion classes we find that these
equations have no solutions.
\end{abstract}

\section{Introduction}

\setcounter{equation}{0}Seven-dimensional manifolds with $G_{2}$-structure
have been studied for more than 40 years. Already in 1969, Alfred Gray
studied vector cross products on manifolds \cite{Gray-VCP}, which on $7$%
-manifolds do actually correspond to $G_{2}$-structures. Later on, Fern\'{a}%
ndez and Gray classified the possible torsion classes of $G_{2}$-structures 
\cite{FernandezGray}. The concept of $G_{2}$-structures provides a
classification for a large class of $7$-manifolds. In fact, it is well-known
that a $7$-manifolds admits a $G_{2}$-structure if and only if the first two
Stiefel-Whitney classes $w_{1}$ and $w_{2}$ vanish \cite{FernandezGray,
FriedrichNPG2}. Alternatively, a $7$-manifold admits a $G_{2}$-structure if
and only if it is orientable and admits a spin structure. A very important
special case of a $G_{2}$-structure is when the torsion vanishes. This
implies that the holonomy group lies in $G_{2}$. In Section \ref{secg2struct}
we give a more precise definition and an overview of the properties of $G_{2}
$-structures.

Suppose we are given a $7$-manifold that admits a $G_{2}$-structure, we can
ask the question - $G_{2}$-structures of which torsion classes exist on it?
This is of course a very difficult question, and it is still not clear how
to approach this. However, we could start with some given $G_{2}$-structure,
deform it and then require that the new $G_{2}$-structure lies in some
particular torsion class. This is precisely what we attempt in this paper.
In Section \ref{sectorsdeform} we first derive an expression for the $G_{2}$%
-structure torsion for a general (non-infinitesimal) deformation, and then
in Section \ref{seclam7deform}, we specialize to a particular type of
deformation - deformations that are defined by a vector (that is, the $G_{2}$
invariant $3$-form is deformed by a $3$-form lying in the $7$-dimensional
component of $\Lambda ^{3}$). In this case, we obtain an explicit expression
for the torsion of the deformed $G_{2}$-structure in terms of the old $G_{2}$%
-structure, its torsion and the vector which defines the deformation. We
then proceed to show that a deformation of this type takes a torsion-free $%
G_{2}$-structure to another torsion-free $G_{2}$-structure if and only if
the vector that defines the deformation is parallel. Moreover we also show
that on closed, compact manifolds there are no such deformations from strict
torsion classes $W_{1}$, $W_{7}$, $W_{1}\oplus W_{7}$ to the vanishing
torsion class $W_{0}$, and vice versa.

Such deformations of $G_{2}$-structures have been first considered by
Karigiannis in \cite{karigiannis-2005-57}, where he wrote down the deformed
metric and Hodge star operation, and indeed asked the question whether it is
possible to deform a $G_{2}$-structure to a strictly smaller torsion class.
This paper aims to give a partial answer to this question. Here we are
mainly concerned with non-infinitesimal deformations, but infinitesimal
deformations and flows of $G_{2}$-structures, and in particular properties
of the moduli space of manifolds with $G_{2}$ holonomy have also been
studied by Karigiannis \cite{karigiannis-2005-57, karigiannis-2007},
Karigiannis and Leung \cite{karigiannis-2007a} and by Grigorian and Yau \cite%
{GrigorianG2Review, GrigorianYau1}.

\section{$G_{2}$-structures}

\setcounter{equation}{0}\label{secg2struct}The 14-dimensional group $G_{2}$
is the smallest of the five exceptional Lie groups and is closely related to
the octonions. In particular, $G_{2}$ can be defined as the automorphism
group of the octonion algebra. The restriction of octonion multiplication to
just the imaginary octonions defines a vector cross product on $V=\mathbb{R}%
^{7}$ and vice versa. Moreover, the Euclidean inner product on $V$ can also
be defined in terms of octonion multiplication. The group that preserves the
vector cross product is precisely $G_{2}$ and since it preserves the inner
product as well, we can see that it is a subgroup of $SO\left( 7\right) $.
For more on the relationship between octonions and $G_{2}$, see \cite%
{BaezOcto, GrigorianG2Review}.The structure constants of the vector cross
product define a particular $3$-form on $\mathbb{R}^{7}$, hence $G_{2}$ can
alternatively be defined in the following way.

\begin{definition}
Let $\left( e^{1},e^{2},...,e^{7}\right) $ be a basis for $V^{\ast }$, and
denote $e^{i}\wedge e^{j}\wedge e^{k}$ by $e^{ijk}$. Then define $\varphi
_{0}$ to be the $3$-form on $\mathbb{R}^{7}$ given by 
\begin{equation}
\varphi _{0}=e^{123}+e^{145}+e^{167}+e^{246}-e^{257}-e^{347}-e^{356}.
\label{phi0def}
\end{equation}%
Then $G_{2}$ is defined as the subgroup of $GL\left( 7,\mathbb{R}\right) $
which preserves $\varphi _{0}$.
\end{definition}

Suppose for some $3$-form $\varphi $ on $V$ we define a bilinear form by 
\begin{equation}
B_{\varphi }\left( u,v\right) =\frac{1}{6}\left( u\lrcorner \varphi \right)
\wedge \left( v\lrcorner \varphi \right) \wedge \varphi  \label{Bphi}
\end{equation}%
Here the symbol $\lrcorner $ denotes contraction of a vector with the
differential form: 
\begin{equation*}
\left( u\lrcorner \varphi \right) _{mn}=u^{a}\varphi _{amn}.
\end{equation*}%
Note that we will also use this symbol for contractions of differential
forms using the metric. So for a $p$-form $\alpha $ and a $\left( p+q\right) 
$-form $\beta $, for $q\geq 0$, 
\begin{equation}
\left( \alpha \lrcorner \beta \right) _{b_{1}...b_{q}}=\alpha
^{a_{1}...a_{p}}\beta _{a_{1}...a_{p}b_{1}...b_{q}}  \label{inpdef}
\end{equation}%
where the indices on $\alpha $ are raised using the metric.

Following Hitchin (\cite{Hitchin:2000jd}), $B_{\varphi }$ is a symmetric
bilinear form on $V$ with values in the one-dimensional space $\Lambda
^{7}V^{\ast }$. Hence it defines a linear map $K_{\varphi }:V\longrightarrow
V^{\ast }\times \Lambda ^{7}V^{\ast }.$ Then taking the determinant we get $%
\det K_{\varphi }\in \left( \Lambda ^{7}V^{\ast }\right) ^{9}$, so if this
does not vanish, we choose a positive root - $\left( \det K_{\varphi
}\right) ^{\frac{1}{9}}\in \Lambda ^{7}V^{\ast }$. Then we obtain an inner
product on $V$%
\begin{equation}
g_{\varphi }\left( u,v\right) =B_{\varphi }\left( u,v\right) \left( \det
K_{\varphi }\right) ^{-\frac{1}{9}}  \label{gphi1}
\end{equation}%
and the volume form of this inner product is then $\left( \det K_{\varphi
}\right) ^{\frac{1}{9}}$. In components we can rewrite this as 
\begin{equation}
\left( g_{\varphi }\right) _{ab}=\left( \det s\right) ^{-\frac{1}{9}}s_{ab}.
\label{metricdefdirect}
\end{equation}%
with 
\begin{equation}
s_{ab}=\frac{1}{144}\varphi _{amn}\varphi _{bpq}\varphi _{rst}\hat{%
\varepsilon}^{mnpqrst}  \label{sabdef}
\end{equation}%
where $\hat{\varepsilon}^{mnpqrst}$ is the alternating symbol with $\hat{%
\varepsilon}^{12...7}=+1$.

Applying (\ref{Bphi}) to $\varphi _{0}$ as defined by (\ref{phi0def}), we
recover the standard Euclidean metric on $V$:%
\begin{equation}
g_{0}=\left( e^{1}\right) ^{2}+...+\left( e^{7}\right) ^{2}.  \label{g0def}
\end{equation}%
As we know, the stabilizer of $\varphi _{0}$ in $GL\left( 7,\mathbb{R}%
\right) $ is $G_{2}$, which is $14$-dimensional. Since $GL\left( 7,\mathbb{R}%
\right) $ is $49$-dimensional, we find that the orbit of $\varphi _{0}$ in $%
\Lambda ^{3}V^{\ast }$ has dimension $49-14=35=\dim \Lambda ^{3}V^{\ast }$.
Hence the orbit of $\varphi _{0}$ is an open subset $\Lambda _{+}^{3}$ $%
\subset \Lambda ^{3}V^{\ast }$.

\begin{definition}
\label{defpos3form}Let $V$ be a $7$-dimensional real vector space. Then a $3$%
-form $\varphi $ is said to be \emph{positive }if lies in the $GL\left( 7,%
\mathbb{R}\right) $ orbit of $\varphi _{0}$.
\end{definition}

In fact, in $\Lambda ^{3}V^{\ast }$ there are two open orbits of $GL\left( 7,%
\mathbb{R}\right) $ \cite{Bryant-1987}. The second open orbit consists of $3 
$-forms for which the metric defined by (\ref{gphi1}) has indefinite
signature $\left( 4,3\right) $, and the corresponding stabilizer is the
so-called \emph{split }$G_{2}$. The $3$-form that it stabilizes can be
obtained by changing the minus signs to plus signs in the expression (\ref%
{phi0def}) for $\varphi _{0}$. The existence of these open orbits gives a
notion of a \emph{non-degenerate} $3$-form on $V$ - that is, a $3$-form
which lies in one of the open orbits \cite{Hitchin:2000jd}. Moreover, it
turns out that non-degeneracy of a $3$-form is equivalent to non-degeneracy
of the corresponding metric. Thus if the determinant of the metric (\ref%
{metricdefdirect}), or equivalently $\det \left( s_{ab}\right) $ for $s_{ab}$
in (\ref{sabdef}), is non-zero, then the $3$-form is in one of the open
orbits. If moreover, the metric is positive-definite, then the $3$-form is
positive.

Now, given a $n$-dimensional manifold $M$, a $G$-structure on $M$ for some
Lie subgroup $G$ of $GL\left( n,\mathbb{R}\right) $ is a reduction of the
frame bundle $F$ over $M$ to a principal subbundle $P$ with fibre $G$. The
concept of a $G$-structure gives a convenient way of encoding different
geometric structures. For example, an $O\left( n\right) $-structure is a
reduction of the frame bundle to a subbundle with fibre $O\left( n\right) .$
This defines an orthonormal frame\emph{\ }at each point at $M$ and thus we
can define a Riemannian metric on $M$. Hence there is a $1$-$1$
correspondence between $O\left( n\right) $-structures and Riemannian
metrics. Similarly, an almost complex structure on a $2m$-dimensional
manifold $M$ is equivalent to a $GL\left( m,\mathbb{C}\right) $-structure.

A $G_{2}$-structure is then a reduction of the frame bundle on a $7$%
-dimensional manifold $M$ to a $G_{2}$ principal subbundle. It turns out
that there is a $1$-$1$ correspondence between $G_{2}$-structures and
positive $3$-forms on the manifold. Define the bundle of positive $3$-forms
on $M$ as the subset of $3$-forms $\varphi $ in $\Lambda ^{3}T^{\ast }M$
such that for every point $p$ in $M$, $\left. \varphi \right\vert _{p}\in
\Lambda ^{3}T_{p}^{\ast }M$ is a positive $3$-form in the sense of
Definition \ref{defpos3form}. Using the $G_{2}$ principal bundle we can then
define a positive $3$-form $\varphi $ on the whole manifold. Conversely,
suppose we are given a positive $3$-form. Then at each point $p$ the subset
of $GL\left( 7,\mathbb{R}\right) $ that identifies $\left. \varphi
\right\vert _{p}$ with $\varphi _{0}$ is isomorphic to $G_{2}.$ Overall the
whole manifold this will be a subset of the frame bundle $F$, and it is easy
to show that this does give a principal subbundle with fibre $G_{2}$, and
hence a $G_{2}$-structure.

Once we have a $G_{2}$-structure, since $G_{2}$ is a subgroup of $SO\left(
7\right) $, we can also define a Riemannian metric. More concretely, if $%
\varphi $ is the $3$-form which defines the $G_{2}$-structure, then using (%
\ref{gphi1}) we can define a corresponding metric $g$. Following Joyce (\cite%
{Joycebook}), we will adopt the following abuse of notation

\begin{definition}
Let $M$ be an oriented $7$-manifold. The pair $\left( \varphi ,g\right) $
for a positive $3$-form $\varphi $ and corresponding metric $g$ will be
referred to as a $G_{2}$-structure.
\end{definition}

Since a $G_{2}$-structure defines a metric, it also defines a Hodge star.
Thus we can construct another $G_{2}$-invariant object - the $4$-form $\ast
\varphi $. Since the Hodge star is defined by the metric, which in turn is
defined by $\varphi $, the $4$-form $\ast \varphi $ depends non-linearly on $%
\varphi $. For convenience we will usually denote $\ast \varphi $ by $\psi $%
. On $\mathbb{R}^{7}$, when $\varphi $ is given by its canonical form $%
\varphi _{0}$ (\ref{phi0def}), $\psi $ takes the following canonical form%
\begin{equation}
\psi _{0}=e^{4567}+e^{2367}+e^{2345}+e^{1357}-e^{1346}-e^{1256}-e^{1247}.
\label{sphi0def}
\end{equation}%
By considering the canonical forms $\varphi _{0}$ and $\psi _{0}$, we can
write down various contraction identities for a $G_{2}$-structure $\left(
\varphi ,g\right) $ and its corresponding $4$-form $\psi $ \cite%
{bryant-2003,GrigorianYau1,karigiannis-2007}.

\begin{proposition}
\label{propcontractions}The $3$-form $\varphi $ and the corresponding $4$%
-form $\psi $ satisfy the following identities: 
\begin{subequations}%
\label{contids} 
\begin{eqnarray}
\varphi _{abc}\varphi _{mn}^{\ \ \ \ c} &=&g_{am}g_{bn}-g_{an}g_{bm}+\psi
_{abmn}  \label{phiphi1} \\
\varphi _{abc}\psi _{mnp}^{\ \ \ \ \ \ c} &=&3\left( g_{a[m}\varphi
_{np]b}-g_{b[m}\varphi _{np]a}\right)  \label{phipsi} \\
\psi _{abcd}\psi ^{mnpq} &=&24\delta _{a}^{[m}\delta _{b}^{n}\delta
_{c}^{p}\delta _{d}^{q]}+72\psi _{\lbrack ab}^{\ \ \ \ [mn}\delta
_{c}^{p}\delta _{d]}^{q]}-16\varphi _{\lbrack abc}\varphi ^{\ [mnp}\delta
_{d]}^{q]}  \label{psipsi0}
\end{eqnarray}%
\end{subequations}%
where $\left[ m\ n\ p\right] $ denotes antisymmetrization of indices and $%
\delta _{a}^{b}$ is the Kronecker delta, with $\delta _{b}^{a}=1$ if $a=b$
and $0$ otherwise.
\end{proposition}

The above identities can be of course further contracted - the details can
be found in \cite{GrigorianYau1,karigiannis-2007}. These identities and
their contractions are crucial whenever any calculations involving $\varphi $
and $\psi $ have to be done.

For a general $G$-structure, the spaces of $p$-forms decompose according to
irreducible representations of $G$. Given a $G_{2}$-structure, we have the
following decomposition of $p$-forms:%
\begin{subequations}%

\begin{eqnarray}
\Lambda ^{1} &=&\Lambda _{7}^{1}  \label{l1decom} \\
\Lambda ^{2} &=&\Lambda _{7}^{2}\oplus \Lambda _{14}^{2}  \label{l2decom} \\
\Lambda ^{3} &=&\Lambda _{1}^{3}\oplus \Lambda _{7}^{3}\oplus \Lambda
_{27}^{3}  \label{l3decom} \\
\Lambda ^{4} &=&\Lambda _{1}^{4}\oplus \Lambda _{7}^{4}\oplus \Lambda
_{27}^{4}  \label{l4decom} \\
\Lambda ^{5} &=&\Lambda _{7}^{5}\oplus \Lambda _{14}^{5}  \label{l5decom} \\
\Lambda ^{6} &=&\Lambda _{7}^{6}  \label{l6decom}
\end{eqnarray}

\end{subequations}%
The subscripts denote the dimension of representation and components which
correspond to the same representation are isomorphic to each other. We have
the following characterization of the various components \cite%
{Bryant-1987,bryant-2003}:

\begin{proposition}
Let $M$ be a $7$-manifold with a $G_{2}$-structure $\left( \varphi ,g\right) 
$. Then the components of spaces of $2$-, $3$-, $4$-, and $5$-forms are
given by

\begin{eqnarray*}
\Lambda _{7}^{2} &=&\left\{ \alpha \lrcorner \varphi \text{: }\alpha \in
\Lambda _{7}^{1}\right\} \\
\Lambda _{14}^{2} &=&\left\{ \omega \in \Lambda ^{2}\text{: }\left( \omega
_{ab}\right) \in \mathfrak{g}_{2}\right\} =\left\{ \omega \in \Lambda ^{2}%
\text{: }\omega \lrcorner \varphi =0\right\} \\
\Lambda _{1}^{3} &=&\left\{ f\varphi \text{: }f\in C^{\infty }\left(
M\right) \right\} \\
\Lambda _{7}^{3} &=&\left\{ \alpha \lrcorner \psi \text{: }\alpha \in
\Lambda _{7}^{1}\right\} \\
\Lambda _{27}^{3} &=&\left\{ \chi \in \Lambda ^{3}:\chi
_{abc}=h_{[a}^{d}\varphi _{bc]d}\text{ for }h_{ab}~\text{traceless, symmetric%
}\right\} \\
\Lambda _{1}^{4} &=&\left\{ f\psi \text{: }f\in C^{\infty }\left( M\right)
\right\} \\
\Lambda _{7}^{4} &=&\left\{ \alpha \wedge \varphi \text{: }\alpha \in
\Lambda _{7}^{1}\right\} \\
\Lambda _{27}^{4} &=&\left\{ \chi \in \Lambda ^{4}:\chi
_{abcd}=h_{[a}^{e}\psi _{bcd]e}\text{ for }h_{ab}~\text{traceless, symmetric}%
\right\} \\
\Lambda _{7}^{5} &=&\left\{ \alpha \wedge \psi \text{: }\alpha \in \Lambda
_{7}^{1}\right\} \\
\Lambda _{14}^{5} &=&\left\{ \omega \wedge \varphi \text{: }\omega \in
\Lambda _{14}^{2}\text{ }\right\}
\end{eqnarray*}
\end{proposition}

In particular, we see that the $7$-dimensional component of $\Lambda ^{3}$
is defined by $1$-forms (or equivalently, vectors), and the $27$-dimensional
component is given by traceless symmetric tensors. For convenience, and
following \cite{bryant-2003}, we will adopt the following notation for the
map from symmetric tensors into $\Lambda ^{3}$:%
\begin{equation}
\mathrm{i}_{\varphi }:\mathrm{Sym}^{2}\left( V^{\ast }\right)
\longrightarrow \Lambda ^{3}\ \ \text{given by }\mathrm{i}_{\varphi }\left(
h\right) _{abc}=h_{[a}^{d}\varphi _{bc]d}
\end{equation}%
and similarly, for the map from symmetric tensors into $\Lambda ^{4}$:%
\begin{equation}
\mathrm{i}_{\psi }:\mathrm{Sym}^{2}\left( V^{\ast }\right) \longrightarrow
\Lambda ^{4}\ \ \text{given by }\mathrm{i}_{\psi }\left( h\right)
_{abcd}=h_{[a}^{e}\psi _{bcd]e}.
\end{equation}

It is sometimes useful to be able to find projections of given $p$-form onto
the different components. Here we collect some of these results \cite%
{karigiannis-2007, GrigorianYau1}:

\begin{proposition}
\label{proplamb2proj}Suppose $\omega $ is a $2$-form. Then the projections $%
\pi _{7}\left( \omega \right) $ and $\pi _{14}\left( \omega \right) $ onto $%
\Lambda _{7}^{2}$ and $\Lambda _{14}^{2}$, respectively, are given by 
\begin{subequations}%
\begin{eqnarray}
\pi _{7}\left( \omega \right) &=&\alpha \lrcorner \varphi \ \ \text{where }%
\alpha =\frac{1}{6}\omega \lrcorner \varphi \\
\pi _{14}\left( \omega \right) &=&\frac{2}{3}\omega -\frac{1}{6}\omega
\lrcorner \psi
\end{eqnarray}%
\end{subequations}%
\end{proposition}

\begin{proposition}
Suppose $\chi $ is a $3$-form. Then the projections $\pi _{1}\left( \chi
\right) $, $\pi _{7}\left( \chi \right) $ and $\pi _{27}\left( \chi \right) $
onto $\Lambda _{1}^{3}$, $\Lambda _{7}^{3}$ and $\Lambda _{27}^{3}$,
respectively, are given by 
\begin{subequations}%
\begin{eqnarray}
\pi _{1}\left( \chi \right) &=&a\varphi \ \text{where }a=\frac{1}{42}\chi
\lrcorner \varphi \  \\
\pi _{7}\left( \chi \right) &=&\omega \lrcorner \psi \ \text{where }\omega =-%
\frac{1}{24}\chi \lrcorner \psi \  \\
\pi _{27}\left( \chi \right) &=&\mathrm{i}_{\varphi }\left( h\right) \ \text{%
where }h_{ab}=\frac{3}{4}\chi _{mn(a}\varphi _{b)}^{\ \ mn}-\frac{3}{28}%
\left( \chi \lrcorner \varphi \right) g_{ab}\ 
\end{eqnarray}%
\end{subequations}%
. Similarly, if $\chi $ is a $4$-form the corresponding $4$-form, the
projections are 
\begin{subequations}%
\begin{eqnarray}
\pi _{1}\left( \chi \right) &=&a\psi \ \text{where }a=\frac{1}{168}\chi
\lrcorner \psi \  \\
\pi _{7}\left( \chi \right) &=&\omega \wedge \varphi \ \text{where }\omega =-%
\frac{1}{24}\varphi \lrcorner \chi \  \\
\pi _{27}\left( \chi \right) &=&\mathrm{i}_{\psi }\left( h\right) \ \text{%
where }h_{ab}=-\frac{1}{3}\chi _{mnp(a}\psi _{b)}^{\ \ mnp}+\frac{1}{21}%
\left( \chi \lrcorner \psi \right) g_{ab}\ 
\end{eqnarray}%
\end{subequations}%
.
\end{proposition}

\begin{proposition}
\label{proplamb5proj}Suppose $\eta $ is a $5$-form. Then the projections $%
\pi _{7}\left( \eta \right) $ and $\pi _{14}\left( \eta \right) $ onto $%
\Lambda _{7}^{5}$ and $\Lambda _{14}^{5}$, respectively, are given by 
\begin{subequations}%
\begin{eqnarray}
\pi _{7}\left( \eta \right) &=&\alpha \wedge \psi \ \ \text{where }\alpha =%
\frac{1}{72}\psi \lrcorner \eta \\
\pi _{14}\left( \eta \right) &=&\omega \wedge \varphi \ \text{where }\omega =%
\frac{1}{9}\varphi \lrcorner \eta -\frac{1}{36}\left( \varphi \lrcorner \eta
\right) \lrcorner \psi
\end{eqnarray}%
\end{subequations}%
.
\end{proposition}

\begin{proof}
Consider 
\begin{equation}
\eta =\alpha \wedge \psi +\omega \wedge \varphi  \label{eta5}
\end{equation}%
where $\alpha $ is a $1$-form and $\omega \in \Lambda _{14}^{2}$. Then, 
\begin{equation}
\left( \psi \lrcorner \eta \right) _{m}=5\psi ^{abcd}\alpha _{\lbrack a}\psi
_{bcd]m}+10\psi ^{abcd}\omega _{\lbrack ab}\varphi _{cd]m}
\end{equation}%
Using the contractions between $\varphi $ and $\psi $ from Proposition \ref%
{propcontractions}, we find that 
\begin{eqnarray*}
\left( \psi \lrcorner \eta \right) _{m} &=&72\alpha _{m}+24\varphi _{m}^{\ \
ab}\omega _{ab} \\
&=&72\alpha _{m}
\end{eqnarray*}%
since $\omega \lrcorner \varphi =0$ for $\omega \in \Lambda _{14}^{2}$.
Hence we get the $\pi _{7}$ projection.

Now for $\eta $ as in (\ref{eta5}), from definition of the Hodge star, we
have, 
\begin{eqnarray}
\left( \varphi \lrcorner \eta \right) _{ab} &=&\frac{1}{2}\varepsilon
_{abcde}^{\ \ \ \ \ \ \ \ \ mn}\left( \ast \eta \right) _{mn}\varphi ^{abc} 
\notag \\
&=&3\left( \ast \eta \lrcorner \psi \right) _{ab}  \notag \\
&=&12\left( \pi _{7}\ast \eta \right) _{ab}-6\left( \pi _{14}\ast \eta
\right) _{ab}  \label{eta5star}
\end{eqnarray}%
In particular, from Proposition \ref{proplamb2proj}, 
\begin{equation*}
\pi _{14}\left( \varphi \lrcorner \eta \right) =\frac{2}{3}\left( \varphi
\lrcorner \eta \right) -\frac{1}{6}\left( \varphi \lrcorner \eta \right)
\lrcorner \psi
\end{equation*}%
and so, 
\begin{equation}
\pi _{14}\left( \ast \eta \right) =-\frac{1}{9}\left( \varphi \lrcorner \eta
\right) +\frac{1}{36}\left( \varphi \lrcorner \eta \right) \lrcorner \psi
\end{equation}%
However, 
\begin{eqnarray*}
\ast \left( \pi _{14}\left( \ast \eta \right) \wedge \varphi \right) _{mn}
&=&\frac{1}{12}\varepsilon _{mn}^{\ \ \ \ \ abcde}\pi _{14}\left( \ast \eta
\right) _{ab}\varphi _{cde} \\
&=&\frac{1}{2}\left( \pi _{14}\left( \ast \eta \right) \lrcorner \psi
\right) _{mn} \\
&=&-\pi _{14}\left( \ast \eta \right) _{mn}
\end{eqnarray*}%
Hence, 
\begin{equation}
\pi _{14}\left( \eta \right) =-\pi _{14}\left( \ast \eta \right) \wedge
\varphi .
\end{equation}
\end{proof}

\section{$G_{2}$-structure torsion}

\setcounter{equation}{0}As before, suppose $M$ is a $7$-dimensional manifold
with a $G_{2}$-structure $\left( \varphi ,g\right) $. The metric $g$ defines
a reduction of the frame bundle to a principal $SO\left( 7\right) $%
-subbundle $Q$, that is, a subbundle of oriented orthonormal frames. The
metric also defines a Levi-Civita connection $\nabla $ on the tangent bundle 
$TM$, and hence on $F$. However, the $G_{2}$-invariant $3$-form $\varphi $
reduces the orthonormal bundle further to a principal $G_{2}$-subbundle $Q$.
We can then pull back the Levi-Civita connection to $Q$. On $Q$ we can
uniquely decompose $\nabla $ as 
\begin{equation}
\nabla =\bar{\nabla}+\mathcal{T}  \label{tors}
\end{equation}%
where $\bar{\nabla}$ is a $G_{2}$-compatible \emph{canonical connection }$%
\bar{\nabla}$ on $P$, taking values in $\mathfrak{g}_{2}\subset \mathfrak{so}%
\left( 7\right) $, while $\mathcal{T}$ is a $1$-form taking values in $%
\mathfrak{g}_{2}^{\perp }\subset \mathfrak{so}\left( 7\right) $. This $1$%
-form $\mathcal{T}$ is known as the $\emph{intrinsic}$ \emph{torsion} of the 
$G_{2}$-structure. The intrinsic torsion is precisely the obstruction to the
Levi-Civita connection being $G_{2}$-compatible. Note that $\mathfrak{so}%
\left( 7\right) $ splits according to $G_{2}$ representations as 
\begin{equation*}
\mathfrak{so}\left( 7\right) \cong \Lambda ^{2}V\cong \Lambda _{7}^{2}\oplus
\Lambda _{14}^{2}
\end{equation*}%
but $\Lambda _{14}^{2}\cong \mathfrak{g}_{2}$, so the complement $\mathfrak{g%
}_{2}^{\perp }\cong \Lambda _{7}^{2}\cong V$. Hence $\mathcal{T}$ can be
represented by a tensor $T_{ab}$ which lies in $W\cong V\otimes V$. Now,
since $\varphi $ is $G_{2}$-invariant, it is $\bar{\nabla}$-parallel, so the
torsion is determined by $\nabla \varphi $.

Following \cite{karigiannis-2007}, consider the $3$-form $\nabla _{X}\varphi 
$ for some vector field $X$. It is easy to see 
\begin{equation}
\nabla _{X}\varphi \in \Lambda _{7}^{3}  \label{torsphi37}
\end{equation}%
and thus overall, 
\begin{equation}
\nabla \varphi \in \Lambda _{7}^{1}\otimes \Lambda _{7}^{3}\cong W.
\label{torsphiW}
\end{equation}%
Thus $\nabla \varphi $ lies in the same space as $T_{ab}$ and thus
completely determines it. Given (\ref{torsphiW}), we can write 
\begin{equation}
\nabla _{a}\varphi _{bcd}=T_{a}^{\ \ e}\psi _{ebcd}  \label{fulltorsion}
\end{equation}%
where $T_{ab}$ is the \emph{full torsion tensor}. From this we can also
write 
\begin{equation}
T_{a}^{\ m}=\frac{1}{24}\left( \nabla _{a}\varphi _{bcd}\right) \psi ^{mbcd}.
\label{tamphipsi}
\end{equation}%
This $2$-tensor fully defines $\nabla \varphi $ since pointwise, it has 49
components and the space $W$ is also 49-dimensional (pointwise). In general
we can split $T_{ab}$ into \emph{torsion components }as 
\begin{equation}
T=\tau _{1}g+\tau _{7}\lrcorner \varphi +\tau _{14}+\tau _{27}
\label{torsioncomps}
\end{equation}%
where $\tau _{1}$ is a function, and gives the $\mathbf{1}$ component of $T$%
. We also have $\tau _{7}$, which is a $1$-form and hence gives the $\mathbf{%
7}$ component, and, $\tau _{14}\in \Lambda _{14}^{2}$ gives the $\mathbf{14}$
component and $\tau _{27}$ is traceless symmetric, giving the $\mathbf{27}$
component. Note that the normalization of these components is different from 
\cite{karigiannis-2007}. Hence we can split $W$ as 
\begin{equation}
W=W_{1}\oplus W_{7}\oplus W_{14}\oplus W_{27}.  \label{Wsplit}
\end{equation}%
Originally the torsion of $G_{2}$-structures was studied by Fern\'{a}ndez
and Gray \cite{FernandezGray}, and their analysis revealed that there are in
fact a total of 16 torsion classes of $G_{2}$-structures. Later on,
Karigiannis reproduced their results using simple computational arguments 
\cite{karigiannis-2007}.The 16 torsion classes arise as the subsets of $W$
which $\nabla \varphi $ belongs to.

Note that our notation differs from Fern\'{a}ndez and Gray. Our $\tau _{1}$
corresponds to their $\tau _{0}$, $\tau _{7}$ corresponds to $\tau _{4},$ $%
\tau _{14}$ corresponds to $\tau _{2}$ and $\tau _{27}$ corresponds to $\tau
_{3}$.

Moreover, as shown in \cite{karigiannis-2007}, the torsion components $\tau
_{i}$ relate directly to the expression for $d\varphi $ and $d\psi $. In
fact, in our notation, 
\begin{subequations}%
\label{dptors} 
\begin{eqnarray}
d\varphi &=&4\tau _{1}\psi -3\tau _{7}\wedge \varphi -\ast \tau _{27}
\label{dphitors} \\
d\psi &=&-4\tau _{7}\wedge \psi -2\ast \tau _{14}.  \label{dpsitors}
\end{eqnarray}
\end{subequations}%
Similarly to (\ref{fulltorsion}), we can express the covariant derivative of 
$\psi $ in terms of $T$.

\begin{lemma}
Given a $G_{2}$-structure defined by $3$-form $\varphi $, with torsion $%
T_{a}^{\ \ m}$ given by (\ref{tamphipsi}), the covariant derivative of the
corresponding $4$-form $\psi $ is given by 
\begin{equation}
\nabla _{a}\psi _{bcde}=-4T_{a[b}\varphi _{cde]}  \label{psitorsion}
\end{equation}
\end{lemma}

\begin{proof}
Consider the identity (\ref{phiphi1}):%
\begin{equation*}
\varphi _{abc}\varphi _{\ \ mn}^{c}=g_{am}g_{bn}-g_{an}g_{bm}+\psi _{abmn}
\end{equation*}%
Applying the covariant derivative to both sides, we get 
\begin{equation*}
\nabla _{e}\psi _{abmn}=\left( \nabla _{e}\varphi _{abc}\right) \varphi _{\
\ mn}^{c}+\varphi _{abc}\left( \nabla _{e}\varphi _{\ \ mn}^{c}\right)
\end{equation*}%
Now using (\ref{fulltorsion}) and using contraction identities between $%
\varphi $ and $\psi $, we get (\ref{psitorsion}).
\end{proof}

Suppose $d\varphi =d\psi =0$. Then this means that all four torsion
components vanish and hence $T=0$, and as a consequence $\nabla \varphi =0$.
The converse is trivially true. This result is originally due to Fern\'{a}%
ndez and Gray \cite{FernandezGray}. Moreover, a $G_{2}$-structure is
torsion-free if and only if the holonomy of the corresponding metric is
contained in $G_{2}$ \cite{Joycebook}.

The torsion tensor $T_{ab}$ and hence the individual components $\tau _{1}$,$%
\tau _{7}$,$\tau _{14}$ and $\tau _{27}$ must also satisfy certain
differential conditions. For the exterior derivative $d$, $d^{2}=0$, so from
(\ref{dptors}), must have 
\begin{eqnarray*}
d\left( 4\tau _{1}\psi -3\tau _{7}\wedge \varphi -\ast \tau _{27}\right) &=&0
\\
d\left( 4\tau _{7}\wedge \psi +2\ast \tau _{14}\right) &=&0
\end{eqnarray*}%
Alternatively, note that we have 
\begin{eqnarray*}
\left( d^{2}\varphi \right) _{abcde} &=&20\nabla _{\lbrack a}\nabla
_{b}\varphi _{cde]} \\
&=&20\nabla _{\lbrack a}T_{b}^{\ \ f}\psi _{|f|cde]}
\end{eqnarray*}%
and 
\begin{eqnarray*}
\left( d^{2}\psi \right) _{abcdef} &=&30\nabla _{\lbrack a}\nabla _{b}\psi
_{cdef]} \\
&=&30\nabla _{\lbrack a}T_{bc}\varphi _{def]}
\end{eqnarray*}%
So in particular, we get conditions 
\begin{subequations}%
\label{dtabcond} 
\begin{eqnarray}
\nabla _{\lbrack a}T_{b}^{\ \ f}\psi _{|f|cde]} &=&0 \\
\nabla _{\lbrack a}T_{bc}\varphi _{def]} &=&0
\end{eqnarray}%
\end{subequations}%
From these, we get the following conditions.

\begin{proposition}
\label{proptabconds}The torsion tensor $T_{ab}$ of a $G_{2}$-structure $%
\varphi $ satisfies the following consistency conditions

\begin{enumerate}
\item 
\begin{equation}
\varphi ^{abc}T_{bc}T_{am}^{\ \ }-T^{bd}T_{\ \ b}^{c}\varphi _{mdc}-\psi
_{m}^{\ \ abc}\nabla _{a}T_{bc}-\left( \func{Tr}T\right) \varphi _{m}^{\ \
ab}T_{ab}=0  \label{dtabcond1}
\end{equation}

\item 
\begin{equation}
\nabla _{m}\left( \func{Tr}T\right) -\nabla _{a}T_{m}^{\ \ \
a}-T_{mc}\varphi ^{abc}T_{ab}=0  \label{dtabcond2}
\end{equation}

\item 
\begin{eqnarray}
0 &=&-\varphi _{mn}^{\ \ \ \ \ c}\nabla _{c}\left( \func{Tr}T\right)
+6T_{a[m}\psi _{n]}^{\ \ abc}T_{bc}+2\left( \nabla _{a}T_{[m\left\vert
b\right\vert }^{\ \ \ }\right) \varphi _{n]}^{\ \ ab}  \label{dtabcond3} \\
&&+2\left( \nabla _{\lbrack m}T_{\left\vert ab\right\vert }\right) \varphi
_{n]}^{\ \ ab}+2\psi _{mnab}T^{ca}T_{bc}+2\left( \func{Tr}T\right) \psi
_{mn}^{\ \ \ \ \ ab}T_{ab}+  \notag \\
&&+\varphi _{mna}\varphi _{bcd}T^{cd}T^{ba}-2\varphi _{mna}\varphi
_{bcd}T^{cd}T^{ab}+2T_{[m}^{\ \ \ a}T_{\left\vert a\right\vert n]}-  \notag
\\
&&4\varphi _{\lbrack m}^{\ \ \ ab}\nabla _{|a}T_{b|n]}-2\left( \func{Tr}%
T\right) T_{\left[ mn\right] }+\varphi _{mn}^{\ \ \ \ a}\nabla _{b}T_{a}^{\
\ b}  \notag
\end{eqnarray}
\end{enumerate}
\end{proposition}

\begin{proof}
Let us first look at $\left( d^{2}\varphi \right) _{abcde}$ $=20\nabla
_{\lbrack a}T_{b}^{\ \ f}\psi _{|f|cde]}$. We have 
\begin{eqnarray*}
\nabla _{a}T_{b}^{\ \ f}\psi _{fcde} &=&(\nabla _{a}T_{b}^{\ \ f})\psi
_{fcde}+T_{b}^{\ \ f}\left( \nabla _{a}\psi _{fcde}\right) \\
&=&(\nabla _{a}T_{b}^{\ \ f})\psi _{fcde}-4T_{b}^{\ \ f}T_{[af}\varphi
_{cde]} \\
&=&(\nabla _{a}T_{b}^{\ \ f})\psi _{fcde}-\frac{4}{5}T_{b}^{\ \
f}T_{[a\left\vert f\right\vert }\varphi _{cde]}+\frac{4}{5}T_{b}^{\ \
f}T_{f[a}\varphi _{cde]}+\frac{12}{5}T_{b}^{\ \ f}T_{[ac}\varphi _{de]f}
\end{eqnarray*}%
Anti-symmetrizing we obtain%
\begin{equation}
\nabla _{\lbrack a}T_{b}^{\ \ f}\psi _{\left\vert f\right\vert cde]}=(\nabla
_{\lbrack a}T_{b}^{\ \ f})\psi _{\left\vert f\right\vert cde]}+\frac{4}{5}%
T_{[a}^{\ \ f}T_{b\left\vert f\right\vert }\varphi _{cde]}-\frac{4}{5}%
T_{[a}^{\ \ f}T_{\left\vert f\right\vert b}\varphi _{cde]}-\frac{12}{5}%
T_{[a}^{\ \ f}T_{bc}\varphi _{de]f}  \label{deltabpsi1}
\end{equation}%
Using Proposition \ref{proplamb5proj}, we find the projections of (\ref%
{deltabpsi1}) onto $\Lambda _{14}^{5}$ and $\Lambda _{7}^{5}$. Considering
the corresponding $2$-form in $\Lambda _{14}^{2}$ we obtain (\ref{dtabcond3}%
). From the $\Lambda _{7}^{5}$ component, we get 
\begin{eqnarray}
0 &=&2\nabla _{m}\left( \func{Tr}T\right) -2\nabla _{a}T_{m}^{\ \ \
a}-2T_{mc}\varphi ^{abc}T_{ab}+\varphi ^{abc}T_{bc}T_{am}^{\ \ }
\label{deltabpsi2} \\
&&-T^{bd}T_{\ \ b}^{c}\varphi _{mdc}-\psi _{m}^{\ \ abc}\nabla
_{a}T_{bc}-\left( \func{Tr}T\right) \varphi _{m}^{\ \ ab}T_{ab}  \notag
\end{eqnarray}%
However, let us now look at $\left( d^{2}\psi \right) _{abcdef}=30\nabla
_{\lbrack a}T_{bc}\varphi _{def]}$ . This is now a $6$-form, so taking the
Hodge star we get $1$-form and hence automatically another $\mathbf{7}$
component. From this we immediately obtain 
\begin{equation*}
\varphi ^{abc}T_{bc}T_{am}^{\ \ }-T^{bd}T_{\ \ b}^{c}\varphi _{mdc}-\psi
_{m}^{\ \ abc}\nabla _{a}T_{bc}-\left( \func{Tr}T\right) \varphi _{m}^{\ \
ab}T_{ab}=0
\end{equation*}%
that is, (\ref{dtabcond1}). Subtracting this condition from (\ref{deltabpsi2}%
), we obtain (\ref{dtabcond2}).
\end{proof}

The Ricci curvature of a $G_{2}$-structure manifold is determined by the
torsion tensor and its derivatives. General expressions for the Ricci
curvature has previously been given by Bryant in \cite{bryant-2003} and
Karigiannis in \cite{karigiannis-2007}. Our expression differs from the
expression in (\cite{karigiannis-2007}) due different sign convention for $%
\psi $. This also leads to a different sign for $T_{ab}$.

\begin{equation}
R_{ab}=\left( \nabla _{a}T_{nm}-\nabla _{n}T_{am}\right) \varphi _{\ \ \
b}^{nm}-T_{an}T_{\ b}^{n}+\func{Tr}\left( T\right) T_{ab}+T_{ac}T_{nm}\psi
_{\ \ \ \ \ b}^{nmc\ \ \ }  \label{torsricci}
\end{equation}%
This expression is a priori not symmetric in indices $a$ and $b$, as it
should be. However we can consider projections of the antisymmetric part of (%
\ref{torsricci}) into the $\mathbf{7}$ and $\mathbf{14}$ representations. It
turns out that the $\mathbf{7}$ component is a combination of (\ref%
{dtabcond1}) and (\ref{dtabcond2}), and hence vanishes. Similarly, the $%
\mathbf{14}$ component is proportional to (\ref{dtabcond3}), and so also
vanishes. Thus indeed, given the conditions (\ref{dtabcond1})-(\ref%
{dtabcond3}), the expression (\ref{torsricci}) for the Ricci curvature is
indeed symmetric. Therefore, the conditions in Proposition \ref{proptabconds}
are very important to ensure overall consistency.

Since we are interested in particular torsion classes, which are given by
torsion components, it is helpful to have conditions corresponding to (\ref%
{dtabcond1})-(\ref{dtabcond3}) in terms of individual torsion components $%
\tau _{1}$, $\tau _{7}$, $\tau _{14}$ and $\tau _{27}\,$.

\begin{proposition}
\label{PropTorsConds}Given the decomposition (\ref{torsioncomps}) of the
full torsion tensor $T_{ab}$ into components $\tau _{1}$, $\tau _{7}$, $\tau
_{14}$ and $\tau _{27}$, these components satisfy the following consistency
conditions:

\begin{enumerate}
\item 
\begin{equation}
\nabla _{a}\left( \tau _{14}\right) _{\ \ m}^{a}+2\varphi _{m}^{\ \ \
ab}\nabla _{a}\left( \tau _{7}\right) _{b}+4\left( \tau _{7}\right)
_{a}\left( \tau _{14}\right) _{\ \ m}^{a}=0  \label{dtabcond1c}
\end{equation}

\item 
\begin{equation}
\nabla _{m}\tau _{1}-\frac{1}{2}\varphi _{m}^{\ \ \ ab}\nabla _{a}\left(
\tau _{7}\right) _{b}-\frac{1}{6}\nabla _{a}\left( \tau _{27}\right) _{\ \
m}^{a}-\left( \tau _{7}\right) _{a}\left( \tau _{27}\right) _{\ \
m}^{a}-\left( \tau _{7}\right) _{m}\tau _{1}=0  \label{dtabcond2c}
\end{equation}

\item 
\begin{eqnarray}
0 &=&\varphi _{mna}\nabla _{b}\left( \tau _{27}\right) ^{ab}+6\nabla
_{a}\left( \tau _{27}\right) _{b[m}\varphi _{n]}^{\ \ ab}-24\tau _{1}\left(
\tau _{14}\right) _{mn}  \label{dtabcond3c} \\
&&-18\left( \frac{2}{3}\nabla _{\lbrack m}\left( \tau _{7}\right) _{n]}-%
\frac{1}{6}\psi _{mn}^{\ \ \ \ ab}\nabla _{a}\left( \tau _{7}\right)
_{b}\right)  \notag \\
&&-18\left( \frac{2}{3}\left( \tau _{14}\right) _{a[m}\left( \tau
_{27}\right) _{n]}^{\ \ a}-\frac{1}{6}\psi _{mn}^{\ \ \ \ ab}\left( \tau
_{14}\right) _{\ \ a}^{c}\left( \tau _{27}\right) _{bc}\right)  \notag
\end{eqnarray}
\end{enumerate}
\end{proposition}

For some of the torsion classes, these conditions simplify, which we
summarize below.%
\begin{equation*}
\begin{tabular}{|l|l|l|}
\hline
\textbf{Torsion class} & \textbf{Condition} & \textbf{In coordinates} \\ 
\hline
${\small W}_{1}$ & ${\small d\tau }_{1}{\small =0}$ & ${\small \nabla }_{m}%
{\small \tau }_{1}{\small =0}$ \\ \hline
${\small W}_{7}$ & ${\small d\tau }_{7}{\small =0}$ & ${\small \nabla }%
_{[m}\left( \tau _{7}\right) _{n]}{\small =0}$ \\ \hline
${\small W}_{14}$ & ${\small d}^{\ast }{\small \tau }_{14}{\small =0}$ & $%
{\small \nabla }_{a}\left( \tau _{14}\right) _{\ \ m}^{a}{\small =0}$ \\ 
\hline
${\small W}_{27}$ & ${\small d}^{\ast }{\small i}_{3}\left( \tau
_{27}\right) {\small =0}$ & 
\begin{tabular}{l}
${\small \nabla }_{a}\left( \tau _{27}\right) _{\ \ m}^{a}{\small =0}$ \\ 
${\small \nabla }_{a}\left( \tau _{27}\right) _{b[m}{\small \varphi }%
_{n]}^{\ \ ab}{\small =0}$%
\end{tabular}
\\ \hline
${\small W}_{1}{\small \oplus W}_{7}$ & 
\begin{tabular}{l}
${\small \tau }_{7}{\small =}d\left( \log \tau _{1}\right) {\small \ }\text{%
if }{\small \tau }_{1}~$nowhere zero \\ 
${\small d\tau }_{7}{\small =0}\text{ and }{\small \tau }_{1}{\small =0\ }%
\text{otherwise}$%
\end{tabular}
&  \\ \hline
${\small W}_{1}{\small \oplus W}_{14}$ & $\tau _{1}=0${\small \ \ or }$\tau
_{14}=0$ &  \\ \hline
${\small W}_{1}{\small \oplus W}_{27}$ & $%
\begin{tabular}{l}
${\small \pi }_{7}\left( d^{\ast }i_{3}\left( \tau _{27}\right) \right) 
{\small =-}\frac{11}{3}\left( d\tau _{1}\right) {\small \lrcorner \varphi }$
\\ 
${\small \pi }_{14}\left( d^{\ast }i_{3}\left( \tau _{27}\right) \right) 
{\small =0}$%
\end{tabular}%
$ & $%
\begin{tabular}{l}
${\small \nabla }_{a}\left( \tau _{27}\right) _{\ \ m}^{a}{\small =6\nabla }%
_{m}{\small \tau }_{1}$ \\ 
${\small \nabla }_{a}\left( \tau _{27}\right) _{b[m}{\small \varphi }%
_{n]}^{\ \ ab}{\small +}\frac{1}{6}{\small \varphi }_{mna}{\small \nabla }%
_{b}\left( \tau _{27}\right) ^{ab}{\small =0}$%
\end{tabular}%
$ \\ \hline
${\small W}_{7}{\small \oplus W}_{14}$ & $%
\begin{tabular}{l}
${\small d\tau }_{7}{\small =0}$ \\ 
${\small d}^{\ast }{\small \tau }_{14}{\small =4\tau }_{7}{\small \lrcorner
\tau }_{14}$%
\end{tabular}%
$ & $%
\begin{tabular}{l}
${\small \nabla }_{[m}\left( \tau _{7}\right) _{n]}{\small =0}$ \\ 
${\small \nabla }_{a}\left( \tau _{14}\right) _{\ \ m}^{a}{\small +4}\left(
\tau _{7}\right) _{a}\left( \tau _{14}\right) _{\ \ m}^{a}{\small =0}$%
\end{tabular}%
$ \\ \hline
\end{tabular}%
\end{equation*}%
To obtain these conditions, we simply use the expressions (\ref{dtabcond1c})
to (\ref{dtabcond3c}), and set the relevant torsion components to zero. For
the class $W_{1}\oplus W_{7}$, the characterization that either $\tau _{7}$
is the gradient of $\log \tau _{1}$ or $\tau _{1}$ is zero everywhere was
given by Cleyton and Ivanov in (\cite{CleytonIvanovConf}).

\section{Deformations of $G_{2}$-structures}

\label{secdeform}\setcounter{equation}{0}Suppose we have a $G_{2}$-structure
on $M$ defined by the $3$-form $\varphi $. In \cite%
{GrigorianYau1,GrigorianG2Review} we considered small deformations of $G_{2}$%
-structures and then expanded related quantities such as the metric $g$, the
volume form $\sqrt{\det g}$ and the $4$-form $\psi $ up to a certain order
in the small parameter. We will now deduce some results about more general
deformations. Suppose we have a deformation for some $3$-form $\chi $ 
\begin{equation}
\varphi \longrightarrow \tilde{\varphi}=\varphi +\chi  \label{phideform1}
\end{equation}%
In \cite{GrigorianYau1,GrigorianG2Review} it was pointed out that
generically it is difficult to obtain a closed form expression for $\tilde{g}
$ and $\tilde{\psi}$. One of the challenges was to obtain a closed form
expression for $\det \tilde{g}$. However it turns out that there is an easy
way to do this, even if obtaining the full explicit expression is still
computationally challenging. Note that we will use upper indices with tilde
to denote indices raised with the deformed metric $\tilde{g}$.

\begin{lemma}
Given a deformation of $\varphi $ as in (\ref{phideform1}), the related
quantities $\tilde{g}$, $\tilde{\psi}$ and $\det g$ are given by: 
\begin{subequations}%
\label{quantdeform}

\begin{eqnarray}
\tilde{g}_{ab} &=&\left( \frac{\det g}{\det \tilde{g}}\right) ^{\frac{1}{2}%
}s_{ab}  \label{gtildedeform1} \\
\tilde{\psi} &=&\left( \frac{\det g}{\det \tilde{g}}\right) ^{\frac{5}{2}}%
\tilde{\eta}  \label{psideform1}
\end{eqnarray}%
\end{subequations}%
where 
\begin{equation}
s_{ab}=g_{ab}+\frac{1}{2}\chi _{mn(a}\varphi _{b)}^{\ \ mn}+\frac{1}{8}\chi
_{amn}\chi _{bpq}\psi ^{mnpq}+\frac{1}{24}\chi _{amn}\chi _{bpq}\left( \ast
\chi \right) ^{mnpq}  \label{sabgen}
\end{equation}%
Moreover, 
\begin{equation}
\tilde{\psi}^{\tilde{a}\tilde{b}\tilde{c}\tilde{d}}=\tilde{g}^{am}\tilde{g}%
^{bn}\tilde{g}^{cp}\tilde{g}^{dq}\tilde{\psi}_{mnpq}=\left( \frac{\det g}{%
\det \tilde{g}}\right) ^{\frac{1}{2}}\left( \psi ^{mnpq}+\ast \chi
^{mnpq}\right)  \label{psiraised}
\end{equation}
\end{lemma}

\begin{proof}
From \cite{GrigorianYau1} we know the expression (\ref{gtildedeform1}) with $%
\tilde{s}_{ab}$ as in (\ref{sabgen}). Now, 
\begin{eqnarray*}
\tilde{\psi}_{abcd} &=&\tilde{\ast}\left( \varphi +\chi \right) _{abcd} \\
&=&\frac{1}{3!}\frac{1}{\sqrt{\det \tilde{g}}}\hat{\varepsilon}%
^{mnpqrst}\left( \varphi _{rst}+\chi _{rst}\right) \tilde{g}_{ma}\tilde{g}%
_{nb}\tilde{g}_{pc}\tilde{g}_{qd}
\end{eqnarray*}%
Here $\tilde{\varepsilon}^{abcdrst}$ refers to the alternating symbol which
takes values $0$ and $\pm 1$. Hence, using (\ref{gtildedeform1}), we get 
\begin{eqnarray}
\tilde{\psi}_{abcd} &=&\left( \frac{\det g}{\det \tilde{g}}\right) ^{\frac{1%
}{2}}\left( \psi ^{mnpq}+\ast \chi ^{mnpq}\right) \tilde{g}_{ma}\tilde{g}%
_{nb}\tilde{g}_{pc}\tilde{g}_{qd}  \label{psitildedeform1a} \\
&=&\left( \frac{\det g}{\det \tilde{g}}\right) ^{\frac{5}{2}}\left( \psi
^{mnpq}+\ast \chi ^{mnpq}\right) s_{ma}s_{nb}s_{pc}s_{qd}  \notag
\end{eqnarray}%
which is precisely (\ref{psideform1}). Incidentally, by raising indices in (%
\ref{psitildedeform1a}) using $\tilde{g}$, we obtain (\ref{psiraised}).
\end{proof}

Another possible simplification is to use contraction formulae for $\tilde{%
\varphi}$ and $\tilde{\psi}$. Since $\tilde{\varphi}$ defines a $G_{2}$%
-structure, $\tilde{\varphi}$ and $\tilde{\psi}$ satisfy the same identities
as $\varphi $ and $\psi $ in Proposition \ref{propcontractions}. So, in
particular, we have%
\begin{eqnarray}
\tilde{\varphi}_{a}^{\ \ \tilde{b}\tilde{c}} &=&\frac{1}{4}\tilde{\varphi}%
_{amn}\tilde{\psi}^{\tilde{m}\tilde{n}\tilde{b}\tilde{c}}=\frac{1}{4}\left( 
\frac{\det g}{\det \tilde{g}}\right) ^{\frac{1}{2}}\left( \varphi
_{amn}+\chi _{amn}\right) \left( \psi ^{mnbc}+\ast \chi ^{mnbc}\right) 
\notag \\
&=&\frac{1}{4}\left( \frac{\det g}{\det \tilde{g}}\right) ^{\frac{1}{2}%
}\left( 4\varphi _{a}^{\ \ bc}+\varphi _{amn}\ast \chi ^{mnbc}+\chi
_{amn}\psi ^{mnbc}+\chi _{amn}\ast \chi ^{mnbc}\right)
\label{phitild2raised}
\end{eqnarray}%
This expression is very simple from computational point of view, since it
does not involve the quantity $s_{ab}$ (apart from the determinant factor).
In our example for $\chi \in \Lambda _{7}^{3}$, this becomes 
\begin{equation}
\tilde{\varphi}_{a}^{\ \ \tilde{b}\tilde{c}}=\left( 1+M\right) ^{-\frac{2}{3}%
}\left( \varphi _{a}^{\ bc}-\chi _{a}^{\ \ bc}+v^{b}v_{m}\varphi _{a\ \ \
}^{\ \ cm}-v^{c}v_{m}\varphi _{a\ \ \ \ }^{\ \ bm}\right)
\label{phitildpi7raised2}
\end{equation}%
We can use (\ref{phitild2raised}) together with other contraction identities
to get closed expressions for the inverse metric $\tilde{g}^{\tilde{a}\tilde{%
b}}$ and the determinant $\det \tilde{g}$.

\begin{proposition}
Given a deformation of $\varphi $ as in (\ref{phideform1}), and the
corresponding deformation of the metric (\ref{gtildedeform1}), the deformed
inverse metric is given by 
\begin{equation}
\tilde{g}^{\tilde{a}\tilde{m}}=\left( \frac{\det g}{\det \tilde{g}}\right)
\gamma ^{am}  \label{gtildinv}
\end{equation}%
where 
\begin{eqnarray}
\gamma ^{am} &=&g^{am}-\frac{1}{96}\ast \chi _{\ \ bcd}^{a}\ast \chi _{\ \ \
pqr}^{m}\varphi ^{bcp}\varphi ^{dqr}-\frac{1}{48}\ast \chi _{\ \
bcd}^{(a}\ast \chi _{\ \ \ pqr}^{m)}\varphi ^{bcp}\chi ^{dqr}
\label{gammaam} \\
&&-\frac{1}{48}\ast \chi _{\ \ bcd}^{(a}\psi _{\ \ \ pqr}^{m)}\chi
^{bcp}\chi ^{dqr}-\frac{1}{96}\psi _{\ \ bcd}^{a}\psi _{\ \ \ pqr}^{m}\chi
^{bcp}\chi ^{dqr}-\frac{1}{96}\ast \chi _{\ \ bcd}^{a}\ast \chi _{\ \ \
pqr}^{m}\chi ^{bcp}\chi ^{dqr}  \notag \\
&&-\frac{1}{4}\chi _{\ \ \ bc}^{(a}\varphi _{\ \ }^{m)bc}+\frac{1}{6}\ast
\chi _{\ \ bcd}^{(a}\varphi _{\ \ \ \ \ e}^{m)b}\chi ^{cde}+\frac{1}{12}\ast
\chi _{\ \ bcd}^{(a}\psi ^{m)bcd}+\frac{1}{12}g^{am}\chi ^{bcd}\varphi _{bcd}
\notag
\end{eqnarray}%
and 
\begin{equation}
\left( \frac{\det \tilde{g}}{\det g}\right) ^{\frac{3}{2}}=\frac{1}{7}\gamma
^{am}s_{am}  \label{detgtildgen}
\end{equation}%
for $s_{am}$ as in (\ref{sabgen}).
\end{proposition}

\begin{proof}
We have the following $G_{2}$-structure contraction identity for $\tilde{%
\varphi}$: 
\begin{equation*}
\tilde{\varphi}_{abc}\tilde{\varphi}_{m}^{\ \ \ \tilde{b}\tilde{c}}=6\tilde{g%
}_{am}
\end{equation*}%
Hence, 
\begin{equation*}
\tilde{\varphi}_{c\ \ \ }^{\ \ \tilde{a}\tilde{b}}\tilde{\varphi}_{b\ }^{\ 
\tilde{m}\tilde{c}}=-6\tilde{g}^{\tilde{a}\tilde{m}}
\end{equation*}%
From (\ref{phitild2raised}), we thus have 
\begin{eqnarray*}
\tilde{g}^{\tilde{a}\tilde{m}} &=&-\frac{1}{6}\tilde{\varphi}_{c\ \ \ }^{\ \ 
\tilde{a}\tilde{b}}\tilde{\varphi}_{b\ }^{\ \tilde{m}\tilde{c}} \\
&=&-\frac{1}{96}\left( \frac{\det g}{\det \tilde{g}}\right) \left( 4\varphi
_{c}^{\ \ ab}+\varphi _{cmn}\ast \chi ^{mnab}+\chi _{cmn}\psi ^{mnab}+\chi
_{cmn}\ast \chi ^{mnab}\right) \times \\
&&\left( 4\varphi _{b}^{\ \ mc}+\varphi _{bpq}\ast \chi ^{pqmc}+\chi
_{bpq}\psi ^{pqmc}+\chi _{bpq}\ast \chi ^{pqmc}\right)
\end{eqnarray*}%
Expanding this, and simplifying using $G_{2}$ contraction identities, we get
(\ref{gtildinv}). Now note that 
\begin{equation*}
\tilde{g}^{\tilde{a}\tilde{m}}\tilde{g}_{am}=\left( \frac{\det g}{\det 
\tilde{g}}\right) ^{\frac{1}{2}}\tilde{g}^{\tilde{a}\tilde{m}}s_{am}=7
\end{equation*}%
with $s_{am}$ given by (\ref{sabgen}). Hence, 
\begin{equation*}
\left( \frac{\det \tilde{g}}{\det g}\right) ^{\frac{3}{2}}=\frac{1}{7}\gamma
^{am}s_{am}.
\end{equation*}
\end{proof}

The simplest deformation would be one where $\chi $ lies in $\Lambda
_{1}^{3} $, and is hence proportional to $\varphi $. So suppose 
\begin{equation}
\chi =\left( f^{3}-1\right) \varphi .  \label{chiconf}
\end{equation}%
This way, 
\begin{equation*}
\tilde{\varphi}=f^{3}\varphi .
\end{equation*}%
From (\ref{sabgen}), we get 
\begin{equation}
s_{ab}=f^{9}g_{ab}.  \label{sabconf}
\end{equation}%
Therefore, from (\ref{gtildedeform1}), 
\begin{equation*}
\tilde{g}_{ab}=\left( \frac{\det g}{\det \tilde{g}}\right) ^{\frac{1}{2}%
}f^{9}g_{ab}.
\end{equation*}%
Taking the determinant on both sides, we find that 
\begin{equation}
\det \tilde{g}=f^{14}\det g.  \label{detgconf}
\end{equation}%
So in fact, 
\begin{eqnarray}
\tilde{g}_{ab} &=&f^{2}g_{ab}  \label{gabconf} \\
\tilde{g}^{ab} &=&f^{-2}g^{ab}
\end{eqnarray}%
and so this defines a conformation transformation. We can also then show
that 
\begin{equation}
\tilde{\psi}=f^{4}\psi .  \label{psiconf}
\end{equation}

The next simplest case is when $\chi =v\lrcorner \psi \in \Lambda _{7}^{3}$.
Then $\ast \chi =-v^{\flat }\wedge \varphi $. From this we can obtain $%
s_{ab} $ and as it was shown by Karigiannis in \cite{karigiannis-2005-57}, 
\begin{equation}
s_{ab}=\left( 1+M\right) g_{ab}-v_{a}v_{b}  \label{sab7}
\end{equation}%
where $M=\left\vert v\right\vert ^{2}$ with the norm taken using $g_{ab}$.
Also, using the expressions for $\chi $ and $\ast \chi $ we can now get $%
\gamma ^{ab}$ from (\ref{gammaam}):%
\begin{equation}
\gamma ^{ab}=\left( 1+M\right) \left( g^{ab}+v^{a}v^{b}\right)
\label{gammamup7}
\end{equation}%
Substituting this into (\ref{detgtildgen}), we get: 
\begin{equation}
\left( \frac{\det \tilde{g}}{\det g}\right) ^{\frac{3}{2}}=\left( 1+M\right)
^{2}  \label{p7detg}
\end{equation}%
and therefore, 
\begin{subequations}%
\label{p7gabbasics} 
\begin{eqnarray}
\tilde{g}_{ab} &=&\left( \frac{\det g}{\det \tilde{g}}\right) ^{\frac{1}{2}%
}s_{ab}=\left( 1+M\right) ^{-\frac{2}{3}}\left( \left( 1+M\right)
g_{ab}-v_{a}v_{b}\right)  \label{p7gabdef} \\
\tilde{g}^{ab} &=&\left( \frac{\det g}{\det \tilde{g}}\right) \gamma
^{am}=\left( 1+M\right) ^{-\frac{1}{3}}\left( g^{mu}+v^{m}v^{u}\right)
\label{p7gabdefinv}
\end{eqnarray}
\end{subequations}%
As expected, these are precisely the results obtained in \cite%
{karigiannis-2005-57}. However the method used here does not depend on the
particular form of $s_{ab}$ and hence theoretically is applicable in the
case when $\chi \in \Lambda _{27}^{3}$ as well.

In fact, when $\chi \in \Lambda _{27}^{3}$, we can write it as 
\begin{equation}
\chi _{abc}=h_{[a}^{\ \ d}\varphi _{bc]d}  \label{chi27}
\end{equation}%
for some traceless, symmetric $h_{ab}$. Then it can be shown that 
\begin{equation}
\ast \chi _{abcd}=-\frac{4}{3}h_{\ [a}^{e}\psi _{\left\vert e\right\vert
bcd]}.  \label{chi427}
\end{equation}%
Now we can substitute both $\chi $ and $\ast \chi $ into the expression (\ref%
{sabgen}) for $s_{ab}$, and after some manipulations get%
\begin{eqnarray}
s_{ab} &=&g_{ab}+\frac{2}{3}h_{ab}+\frac{2}{9}h_{a}^{\ \ c}h_{cb}-\frac{1}{18%
}\func{Tr}\left( h^{2}\right) g_{ab}-\frac{1}{18}\varphi _{amn}\varphi
_{bpq}h^{mn}h^{nq}  \label{chi27sab} \\
&&-\frac{1}{27}\varphi _{amn}\varphi _{bpq}h^{mr}h_{\ \ r}^{p}h^{nq}+\frac{1%
}{81}\func{Tr}\left( h^{3}\right) g_{ab}  \notag
\end{eqnarray}%
This expression for $s_{ab}$ is already rather complicated, and so even
getting an explicit closed expression for $\left( \frac{\det \tilde{g}}{\det
g}\right) $ becomes a very tough task, which is beyond the scope of this
paper.

Now let us consider what happens to the Levi-Civita connection of the
deformed metric.

\begin{lemma}
\label{lemdelconn}Given a deformation of $\varphi $ as in (\ref{phideform1}%
), and the corresponding deformation of the metric (\ref{gtildedeform1}),
the components of the Levi-Civita connection $\tilde{\Gamma}_{a\ \ \ c}^{\ \
b}$ corresponding to the new metric $\tilde{g}$ are given by 
\begin{equation*}
\tilde{\Gamma}_{a\ \ \ c}^{\ \ b}=\Gamma _{a\ \ \ c}^{\ \ b}+\delta \Gamma
_{a\ \ \ c}^{\ \ b}
\end{equation*}%
where%
\begin{eqnarray}
\delta \Gamma _{a\ \ \ c}^{\ \ b} &=&\frac{1}{2}\left( \frac{\det g}{\det 
\tilde{g}}\right) ^{\frac{1}{2}}\left( \tilde{g}^{\tilde{b}\tilde{d}}\left(
\nabla _{c}s_{ad}+\nabla _{a}s_{cd}-\nabla _{d}s_{ac}\right) \right.
\label{deltagamma} \\
&&\left. -\frac{1}{9}\left( \delta _{a}^{b}\delta _{c}^{e}+\delta
_{c}^{b}\delta _{a}^{e}-\tilde{g}_{ac}\tilde{g}^{\tilde{b}\tilde{e}}\right) 
\tilde{g}^{\tilde{m}\tilde{n}}\nabla _{e}s_{mn}\right)  \notag
\end{eqnarray}%
for $s_{ab}$ given by (\ref{sabgen}).
\end{lemma}

\begin{proof}
As it is well known, the components of the Levi-Civita connection are given
by 
\begin{equation}
\Gamma _{a\ \ \ c}^{\ \ b}=\frac{1}{2}g^{bd}\left(
g_{da,c}+g_{dc,a}-g_{ac,d}\right)  \label{LCconn1}
\end{equation}%
and hence for the modified metric, we have%
\begin{equation}
\tilde{\Gamma}_{a\ \ \ c}^{\ \ b}=\frac{1}{2}\tilde{g}^{\tilde{b}\tilde{d}%
}\left( \tilde{g}_{da,c}+\tilde{g}_{dc,a}-\tilde{g}_{ac,d}\right)
\label{LCconn2}
\end{equation}%
The difference between the two connections $\delta \Gamma _{a\ \ \ c}^{\ \
b} $ is then given by 
\begin{equation}
\delta \Gamma _{a\ \ \ c}^{\ \ b}=\tilde{\Gamma}_{a\ \ \ c}^{\ \ b}-\Gamma
_{a\ \ \ c}^{\ \ b}=\frac{1}{2}\tilde{g}^{\tilde{b}\tilde{d}}\left( \nabla
_{c}\tilde{g}_{ad}+\nabla _{a}\tilde{g}_{cd}-\nabla _{d}\tilde{g}_{ac}\right)
\label{LCconndiff}
\end{equation}%
So consider $\nabla _{c}\tilde{g}_{ad}$. Let 
\begin{equation*}
\gamma =\left( \frac{\det \tilde{g}}{\det g}\right) ^{\frac{3}{2}}
\end{equation*}%
then, from (\ref{gtildedeform1}), we have 
\begin{eqnarray*}
\nabla _{c}\tilde{g}_{ad} &=&\nabla _{c}\left( \gamma ^{-\frac{1}{3}%
}s_{ad}\right) =-\frac{1}{3}\gamma ^{-\frac{4}{3}}\left( \nabla _{c}\gamma
\right) s_{ad}+\gamma ^{-\frac{1}{3}}\nabla _{c}s_{ad} \\
&=&-\frac{1}{3}\gamma ^{-1}\left( \nabla _{c}\gamma \right) \tilde{g}%
_{ad}+\gamma ^{-\frac{1}{3}}\nabla _{c}s_{ad}
\end{eqnarray*}%
Now let us look at $\nabla _{c}\gamma $. Note that by definition of the
determinant, we get 
\begin{equation*}
\gamma ^{3}=\left( \frac{\det \tilde{g}}{\det g}\right) ^{\frac{9}{2}}=\frac{%
1}{7!}\frac{1}{\det g}\hat{\varepsilon}^{mnpqrst}\hat{\varepsilon}%
^{abcdefg}s_{am}s_{bn}s_{cp}s_{dq}s_{er}s_{fs}s_{gt}.
\end{equation*}%
and similarly, 
\begin{equation*}
\tilde{g}^{mu}=\frac{1}{6!}\left( \frac{\det g}{\det \tilde{g}}\right) ^{4}%
\frac{1}{\det g}\hat{\varepsilon}^{mnpqrst}\hat{\varepsilon}%
^{ubcdefg}s_{bn}s_{cp}s_{dq}s_{er}s_{fs}s_{gt}
\end{equation*}%
So, in particular, 
\begin{eqnarray*}
\nabla _{c}\left( \gamma ^{3}\right) &=&\nabla _{c}\left( \frac{1}{7!}\frac{1%
}{\det g}\hat{\varepsilon}\hat{\varepsilon}sssssss\right) \\
&=&\frac{1}{6!}\frac{1}{\det g}\hat{\varepsilon}\hat{\varepsilon}\left(
\nabla _{c}s\right) ssssss \\
&=&\gamma ^{\frac{8}{3}}\tilde{g}^{\tilde{m}\tilde{n}}\nabla _{c}s_{mn}
\end{eqnarray*}%
Therefore,%
\begin{equation}
\nabla _{c}\gamma =\frac{1}{3}\gamma ^{\frac{2}{3}}\tilde{g}^{\tilde{m}%
\tilde{n}}\nabla _{c}s_{mn}  \label{delgamma}
\end{equation}%
where for brevity we have omitted the contracted indices. Hence, 
\begin{equation}
\nabla _{c}\tilde{g}_{ad}=-\frac{1}{9}\gamma ^{-\frac{1}{3}}\tilde{g}_{ad}%
\tilde{g}^{\tilde{m}\tilde{n}}\nabla _{c}s_{mn}+\gamma ^{-\frac{1}{3}}\nabla
_{c}s_{ad}.  \label{delcgadtild}
\end{equation}%
Substituting (\ref{delcgadtild}) into (\ref{LCconndiff}), we find that%
\begin{equation*}
\delta \Gamma _{a\ \ \ c}^{\ \ b}=\frac{1}{2}\gamma ^{-\frac{1}{3}}\left( 
\tilde{g}^{\tilde{b}\tilde{d}}\left( \nabla _{c}s_{ad}+\nabla
_{a}s_{cd}-\nabla _{d}s_{ac}\right) -\frac{1}{9}\left( \delta _{a}^{b}\delta
_{c}^{e}+\delta _{c}^{b}\delta _{a}^{e}-\tilde{g}_{ac}\tilde{g}^{\tilde{b}%
\tilde{e}}\right) \tilde{g}^{\tilde{m}\tilde{n}}\nabla _{e}s_{mn}\right)
\end{equation*}%
and then after substituting $\gamma ^{-\frac{1}{3}}=\left( \frac{\det g}{%
\det \tilde{g}}\right) ^{\frac{1}{2}}$we get the result.
\end{proof}

For the conformal deformation with $\chi $ given by (\ref{chiconf}), we find
that 
\begin{equation}
\delta \Gamma _{a\ \ \ c}^{\ \ b}=\frac{1}{9}f^{-9}\partial _{e}f\left(
\delta _{c}^{e}\delta _{a}^{b}+\delta _{a}^{e}\delta
_{c}^{b}-g^{be}g_{ac}\right)  \label{confdelgam}
\end{equation}

\section{Torsion deformations}

\label{sectorsdeform}Suppose we have a deformation of $\varphi $ given by (%
\ref{phideform1}). Using the results from Sect. \ref{secdeform}, we can
calculate the deformed torsion.

\begin{lemma}
Given a deformation of $\varphi $ as in (\ref{phideform1}), the full torsion 
$\tilde{T}$ of the new $G_{2}$-structure $\tilde{\varphi}$ is given by 
\begin{eqnarray}
\tilde{T}_{a}^{\ \Tilde{m}} &=&\left( \frac{\det g}{\det \tilde{g}}\right) ^{%
\frac{1}{2}}\left( T_{a}^{\ m}+\frac{1}{24}T_{a}^{\ e}\psi _{ebcd}\ast \chi
^{mbcd}+\frac{1}{24}\psi ^{mbcd}\nabla _{a}\chi _{bcd}\right.
\label{tamtildefull} \\
&&\left. +\frac{1}{24}\nabla _{a}\chi _{bcd}\ast \chi ^{mbcd}\right) -\frac{1%
}{2}\delta \Gamma _{a\ \ \ b}^{\ \ e}\tilde{\varphi}_{e}^{\ \ \tilde{m}%
\tilde{b}}  \notag
\end{eqnarray}%
with $\delta \Gamma _{a\ \ \ b}^{\ \ e}$ given by (\ref{deltagamma}).
\end{lemma}

\begin{proof}
Starting from (\ref{tamphipsi}) for $\tilde{\varphi}$ and $\tilde{\psi}$, we
get 
\begin{eqnarray*}
\tilde{T}_{a}^{\ \Tilde{m}} &=&\frac{1}{24}\left( \tilde{\nabla}_{a}\tilde{%
\varphi}_{bcd}\right) \tilde{\psi}^{\tilde{m}\tilde{b}\tilde{c}\tilde{d}} \\
&=&\frac{1}{24}\left( \nabla _{a}\tilde{\varphi}_{bcd}-3\delta \Gamma _{a\ \
\ b}^{\ \ e}\tilde{\varphi}_{cde}\right) \tilde{\psi}^{\tilde{m}\tilde{b}%
\tilde{c}\tilde{d}} \\
&=&\frac{1}{24}\left( \left( \frac{\det g}{\det \tilde{g}}\right) ^{\frac{1}{%
2}}\nabla _{a}\tilde{\varphi}_{bcd}\ast \tilde{\varphi}^{mbcd}-3\delta
\Gamma _{a\ \ \ b}^{\ \ e}\tilde{\varphi}_{cde}\tilde{\psi}^{\tilde{m}\tilde{%
b}\tilde{c}\tilde{d}}\right)
\end{eqnarray*}%
We can write 
\begin{equation*}
\tilde{\varphi}_{cde}\tilde{\psi}^{\tilde{m}\tilde{b}\tilde{c}\tilde{d}}=4%
\tilde{\varphi}_{e}^{\ \ \tilde{m}\tilde{b}}
\end{equation*}%
and we can expand%
\begin{eqnarray*}
\nabla _{a}\tilde{\varphi}_{bcd}\ast \tilde{\varphi}^{mbcd} &=&\left( \nabla
_{a}\varphi _{bcd}+\nabla _{a}\chi _{bcd}\right) \left( \psi ^{mbcd}+\ast
\chi ^{mbcd}\right) \\
&=&24T_{a}^{\ m}+T_{a}^{\ e}\psi _{ebcd}\ast \chi ^{mbcd} \\
&&+\psi ^{mbcd}\nabla _{a}\chi _{bcd}+\nabla _{a}\chi _{bcd}\ast \chi ^{mbcd}
\end{eqnarray*}%
Hence the result.
\end{proof}

The torsion classes $W_{i}$ were originally defined by the $G_{2}$-structure 
$\varphi $, so once we have deformed $\varphi $ to $\tilde{\varphi}$ we will
also get new torsion classes. Denote the new space by $\tilde{W}$ which
splits as 
\begin{equation}
\tilde{W}=\tilde{W}_{1}\oplus \tilde{W}_{7}\oplus \tilde{W}_{14}\oplus 
\tilde{W}_{27}.
\end{equation}%
The new torsion $\tilde{T}$ should now split as 
\begin{equation}
\tilde{T}_{ab}=\tilde{\tau}_{1}\tilde{g}_{ab}+\left( \tilde{\tau}_{7}^{%
\tilde{\#}}\lrcorner \tilde{\varphi}\right) _{ab}+\left( \tilde{\tau}%
_{14}\right) _{ab}+\left( \tilde{\tau}_{27}\right) _{ab}
\end{equation}%
Note that $\tilde{\tau}_{7}^{\tilde{\#}}$ refers to the vector obtained from
the $1$-form $\tilde{\tau}_{7}$ by raising indices using the deformed
inverse metric $\tilde{g}^{-1}$. In general, determining these new torsion
components $\tilde{\tau}_{i}$ is quite complicated. First we would have to
lower one of the indices in (\ref{tamtildefull}) using $\tilde{g}$ and
extract the different components. It is however easy to extract the the $%
\tilde{W}_{1}$-component directly from (\ref{tamtildefull}) by just
contracting the indices.

\begin{lemma}
The $\tilde{W}_{1}$-component $\tilde{\tau}_{1}$ of the deformed torsion $%
\tilde{T}$ is given by 
\begin{equation*}
\tilde{\tau}_{1}=\left( \frac{\det g}{\det \tilde{g}}\right) ^{\frac{1}{2}%
}\left( \tau _{1}+\frac{1}{168}T_{a}^{\ e}\psi _{ebcd}\ast \chi ^{abcd}+%
\frac{1}{168}\psi ^{abcd}\nabla _{a}\chi _{bcd}+\frac{1}{168}\nabla _{a}\chi
_{bcd}\ast \chi ^{abcd}\right)
\end{equation*}
\end{lemma}

\begin{proof}
Contracting the indices in (\ref{tamtildefull}) we get 
\begin{equation*}
\tilde{T}_{a}^{\ \tilde{a}}=\left( \frac{\det g}{\det \tilde{g}}\right) ^{%
\frac{1}{2}}\left( T_{a}^{\ a}+\frac{1}{24}T_{a}^{\ e}\psi _{ebcd}\ast \chi
^{abcd}+\frac{1}{24}\psi ^{abcd}\nabla _{a}\chi _{bcd}+\frac{1}{24}\nabla
_{a}\chi _{bcd}\ast \chi ^{abcd}\right) -\frac{1}{2}\delta \Gamma _{a\ \ \
b}^{\ \ e}\tilde{\varphi}_{e}^{\ \ \tilde{a}\tilde{b}}
\end{equation*}%
Note that since the Christoffel symbols are symmetric in the bottom two
indices, $\delta \Gamma _{a\ \ \ b}^{\ \ e}$ is also symmetric in $a$ and $b$%
, so 
\begin{equation*}
\delta \Gamma _{a\ \ \ b}^{\ \ e}\tilde{\varphi}_{e}^{\ \ \tilde{a}\tilde{b}%
}=0
\end{equation*}%
Since 
\begin{equation*}
\tilde{T}_{a}^{\ \tilde{a}}=7\tilde{\tau}_{1},
\end{equation*}%
we get the result.
\end{proof}

Consider now what happens to $\tilde{T}_{an}$ with lowered indices.%
\begin{eqnarray}
\tilde{T}_{an} &=&\tilde{T}_{a}^{\ \Tilde{m}}\tilde{g}_{mn}=\left( \frac{%
\det g}{\det \tilde{g}}\right) \left( T_{a}^{\ m}+\frac{1}{24}T_{a}^{\
e}\psi _{ebcd}\ast \chi ^{mbcd}+\frac{1}{24}\psi ^{mbcd}\nabla _{a}\chi
_{bcd}\right.  \label{Tanlow} \\
&&\left. +\frac{1}{24}\nabla _{a}\chi _{bcd}\ast \chi ^{mbcd}\right) s_{mn}-%
\frac{1}{2}\delta \Gamma _{a\ \ \ b}^{\ \ e}\tilde{\varphi}_{en}^{\ \ \ 
\tilde{b}}  \notag
\end{eqnarray}%
Now using the expression for $\delta \Gamma _{a\ \ \ b}^{\ \ e}$ (\ref%
{deltagamma}), we get 
\begin{eqnarray*}
\delta \Gamma _{a\ \ \ b}^{\ \ e}\tilde{\varphi}_{en}^{\ \ \ \tilde{b}} &=&%
\frac{1}{2}\left( \frac{\det g}{\det \tilde{g}}\right) ^{\frac{1}{2}}\left( 
\tilde{g}^{\tilde{e}\tilde{d}}\tilde{\varphi}_{en}^{\ \ \ \tilde{b}}\left(
\nabla _{b}s_{ad}+\nabla _{a}s_{bd}-\nabla _{d}s_{ab}\right) \right. \\
&&\left. -\frac{1}{9}\left( \delta _{a}^{e}\delta _{b}^{f}+\delta
_{b}^{e}\delta _{a}^{f}-\tilde{g}_{ab}\tilde{g}^{\tilde{e}\tilde{f}}\right) 
\tilde{\varphi}_{en}^{\ \ \ \tilde{b}}\tilde{g}^{\tilde{p}\tilde{q}}\nabla
_{f}s_{pq}\right)
\end{eqnarray*}%
Simplifying further, we eventually get%
\begin{eqnarray*}
\delta \Gamma _{a\ \ \ b}^{\ \ e}\tilde{\varphi}_{en}^{\ \ \ \tilde{b}}
&=&\left( \frac{\det g}{\det \tilde{g}}\right) ^{\frac{1}{2}}\left( \tilde{%
\varphi}_{n}^{\ \ \tilde{b}\tilde{d}}\nabla _{b}s_{ad}-\frac{1}{9}\tilde{%
\varphi}_{an}^{\ \ \ \ \tilde{f}}\tilde{g}^{\tilde{p}\tilde{q}}\nabla
_{f}s_{pq}\right) \\
&=&\left( \frac{\det g}{\det \tilde{g}}\right) ^{\frac{1}{2}}\tilde{\varphi}%
_{c}^{\ \ \tilde{b}\tilde{d}}\left( \delta _{n}^{c}\nabla _{b}s_{ad}-\frac{1%
}{9}\delta _{a}^{c}\tilde{g}_{bn}\tilde{g}^{\tilde{p}\tilde{q}}\nabla
_{d}s_{pq}\right) \\
&=&\frac{1}{4}\left( \frac{\det g}{\det \tilde{g}}\right) \left( 4\varphi
_{c}^{\ \ bd}+\varphi _{cmn}\ast \chi ^{mnbd}+\chi _{cmn}\psi ^{mnbd}+\chi
_{cmn}\ast \chi ^{mnbd}\right) \times \\
&&\times \left( \delta _{n}^{c}\nabla _{b}s_{ad}-\frac{1}{9}\delta _{a}^{c}%
\tilde{g}_{bn}\tilde{g}^{\tilde{p}\tilde{q}}\nabla _{d}s_{pq}\right)
\end{eqnarray*}%
Thus, overall we have 
\begin{eqnarray}
\tilde{T}_{an} &=&\frac{1}{24}\left( \frac{\det g}{\det \tilde{g}}\right)
\left( 24T_{a}^{\ m}+T_{a}^{\ e}\psi _{ebcd}\ast \chi ^{mbcd}+\psi
^{mbcd}\nabla _{a}\chi _{bcd}+\nabla _{a}\chi _{bcd}\ast \chi ^{mbcd}\right)
\times  \label{Tanlow2} \\
&&\times s_{mn}-\frac{1}{8}\left( \frac{\det g}{\det \tilde{g}}\right)
\left( 4\varphi _{c}^{\ \ bd}+\varphi _{cpq}\ast \chi ^{pqbd}+\chi
_{cpq}\psi ^{pqbd}+\chi _{cpq}\ast \chi ^{pqbd}\right) \times  \notag \\
&&\times \left( \delta _{n}^{c}\nabla _{b}s_{ad}-\frac{1}{9}\delta _{a}^{c}%
\tilde{g}_{bn}\tilde{g}^{\tilde{p}\tilde{q}}\nabla _{d}s_{pq}\right) . 
\notag
\end{eqnarray}%
In the particular case of conformal deformations we can simply plug in $\chi 
$, $s$ and $\tilde{g}$ as in (\ref{chiconf}), (\ref{sabconf}) and (\ref%
{gabconf}) into (\ref{Tanlow}) and obtain the deformed torsion.

\begin{proposition}
\label{conftorsion}Let $\left( \varphi ,g\right) $ be a $G_{2}$-structure
with torsion $T$. Then define $\left( \tilde{\varphi},\tilde{g}\right) $ to
be a new $G_{2}$-structure given by a conformal transformation of $\left(
\varphi ,g\right) $: 
\begin{eqnarray*}
\tilde{\varphi} &=&f^{3}\varphi \\
\tilde{g} &=&f^{2}g.
\end{eqnarray*}%
Then the full torsion tensor $\tilde{T}$ is given by 
\begin{equation}
\tilde{T}=fT-df\lrcorner \varphi  \label{conftranstors}
\end{equation}
\end{proposition}

Thus from Proposition \ref{conftorsion} we see that a conformal
transformation only affects the $W_{7}$ torsion class, while the torsion
classes in $W_{1}$, $W_{14}$ and $W_{27}$ are simply scaled. Therefore, the
only conformally invariant torsion classes are the ones that contain a $%
W_{7} $ component. This was previously shown in \cite{FernandezGray} and (%
\cite{karigiannis-2005-57}) but here we have an explicit expression for the
torsion from which this conclusion follows trivially.

The expression (\ref{conftranstors}) also shows that if the $W_{7}$
component of the original torsion is an exact form, then it is possible to
remove this component by applying a particular conformal transformation.
Note that this implies that the class $W_{1}\oplus W_{7}$ is conformal to
the class $W_{1}$. As we know from (!!!), if $\tau _{1}\neq 0$, then 
\begin{equation*}
{\small \tau }_{7}{\small =}d\left( \log \tau _{1}\right)
\end{equation*}%
Hence in order to remove this torsion component, need 
\begin{equation*}
d\left( \log \tau _{1}\right) =\frac{1}{f}df
\end{equation*}%
hence, 
\begin{equation*}
f=\tau _{0}\tau _{1}
\end{equation*}%
is a solution for a constant $\tau _{0}$. The original torsion is 
\begin{equation*}
T=\tau _{1}g+\frac{1}{f}df\lrcorner \varphi
\end{equation*}%
so, under the change (\ref{conftranstors}), the new torsion will become 
\begin{equation*}
\tilde{T}=\tau _{0}\tau _{1}^{2}g
\end{equation*}%
However, under the transformation 
\begin{equation}
\varphi \longrightarrow \tau _{1}^{3}\varphi ,  \label{tau1conftrans}
\end{equation}%
the metric changes as 
\begin{equation*}
g\longrightarrow \tau _{1}^{2}g
\end{equation*}%
Hence in terms of the new metric, the new torsion is 
\begin{equation*}
\tilde{T}=\tau _{0}\tilde{g}
\end{equation*}%
and so the constant $\tau _{0}$ is the new $W_{1}$ torsion component. Thus
the conformal transformation (\ref{tau1conftrans}) reduces the class $%
W_{1}\oplus W_{7}$ to $W_{1}$. Conversely, a conformal transformation of the 
$W_{1}$ class will result in $W_{1}\oplus W_{7}$. Since $G_{2}$-structures
in the $W_{1}$ class are sometimes called \emph{nearly }$G_{2}$ or \emph{%
nearly parallel}, the $G_{2}$-structures in the strict $W_{1}\oplus W_{7}$
class are referred to as \emph{conformally nearly parallel}. If $W_{1}=0$,
then we just have the $W_{7}$ class. In this case, we just know that $\tau
_{7}$ is closed. So by Poincar\'{e} Lemma, we can at least locally find a
function $h$ such that $dh=\tau _{7}\,.$ By taking a conformal
transformation with $f=e^{h}$, we can thus locally fully remove the torsion.
Hence the $W_{7}$ class is sometimes called \emph{locally conformally
parallel. }

For the class $W_{1}\oplus W_{7}$ we can also explicitly write out the Ricci
curvature.

\begin{corollary}
\label{corrrict1t7}Suppose the $3$-form $\varphi $ defines a $G_{2}$%
-structure with torsion contained in the class $W_{1}\oplus W_{7}$. The the
Ricci curvature of the corresponding metric is given by%
\begin{equation}
R_{ab}=\left( \nabla ^{c}\left( \tau _{7}\right) _{c}+5\left( \tau
_{7}\right) ^{c}\left( \tau _{7}\right) _{c}+6\tau _{1}^{2}\right)
g_{ab}-5\left( \tau _{7}\right) _{a}\left( \tau _{7}\right) _{b}+5\nabla
_{a}\left( \tau _{7}\right) _{b}  \label{riccit1t7}
\end{equation}
\end{corollary}

\begin{proof}
In the general expression for the Ricci curvature, (\ref{torsricci}),
substitute 
\begin{equation*}
T_{ab}=\tau _{1}g_{ab}+\left( \tau _{7}\right) ^{c}\varphi _{cab}.
\end{equation*}%
We then get 
\begin{eqnarray*}
R_{ab} &=&\left( \nabla ^{c}\left( \tau _{7}\right) _{c}+5\left( \tau
_{7}\right) ^{c}\left( \tau _{7}\right) _{c}+6\tau _{1}^{2}\right)
g_{ab}-5\left( \tau _{7}\right) _{a}\left( \tau _{7}\right) _{b}+5\nabla
_{b}\left( \tau _{7}\right) _{a} \\
&&-\psi _{ab}^{\ \ \ \ cd}\nabla _{c}\left( \tau _{7}\right) _{d}+\varphi
_{ab}^{\ \ \ c}\nabla _{c}\tau _{1}-\tau _{1}\varphi _{ab}^{\ \ \ c}\left(
\tau _{7}\right) _{c}
\end{eqnarray*}%
However using the fact that $d\tau _{7}=0$, and hence that $\nabla
_{a}\left( \tau _{7}\right) _{b}$ is symmetric, and moreover that $\nabla
_{c}\tau _{1}=\tau _{1}\left( \tau _{7}\right) _{c}$, we obtain (\ref%
{riccit1t7}).
\end{proof}

\section{Torsion for $\Lambda _{7}$ deformations}

\label{seclam7deform}Now consider in detail the case when we have a
deformation in $\Lambda _{7}$. Here we have 
\begin{equation}
h_{ab}=v^{c}\varphi _{cab}  \label{h7}
\end{equation}%
Then, 
\begin{eqnarray*}
\chi _{bcd} &=&h_{[b}^{\ \ e}\varphi _{cd]e}=v^{e}\psi _{bcde}^{\ \ \ \ \ \ }
\\
\ast \chi _{mnpq} &=&4v_{[m}\varphi _{npq]}
\end{eqnarray*}%
So we take a $G_{2}$-structure $\varphi $ and deform it to 
\begin{equation}
\tilde{\varphi}=\varphi +v^{e}\psi _{bcde}^{\ \ \ \ \ \ }  \label{g2deform}
\end{equation}%
For convenience, let 
\begin{equation*}
M=\left\vert v\right\vert ^{2}
\end{equation*}%
then, as we know, 
\begin{equation*}
s_{ab}=g_{ab}\left( 1+M\right) -v_{a}v_{b}
\end{equation*}%
We also know that 
\begin{eqnarray*}
\left( \frac{\det \tilde{g}}{\det g}\right) ^{\frac{1}{2}} &=&\left(
1+M\right) ^{\frac{2}{3}} \\
\tilde{g}_{ab} &=&\left( 1+M\right) ^{-\frac{2}{3}}\left( g_{ab}\left(
1+M\right) -v_{a}v_{b}\right) \\
\tilde{g}^{\tilde{a}\tilde{b}} &=&\left( 1+M\right) ^{-\frac{1}{3}}\left(
g^{ab}+v^{a}v^{b}\right) \\
\tilde{\varphi}_{a}^{\ \ \tilde{b}\tilde{c}} &=&\left( 1+M\right) ^{-\frac{2%
}{3}}\left( \varphi _{a}^{\ bc}-\chi _{a}^{\ \ bc}+v^{b}v_{m}\varphi _{a\ \
\ }^{\ \ cm}-v^{c}v_{m}\varphi _{a\ \ \ \ }^{\ \ bm}\right)
\end{eqnarray*}%
Note that the deformed metric defined above is always positive definite. To
see this, suppose $\xi ^{a}$ is some vector, then 
\begin{equation}
\tilde{g}_{ab}\xi ^{a}\xi ^{b}=\left( 1+\left\vert v\right\vert ^{2}\right)
^{-\frac{2}{3}}\left( \left\vert \xi \right\vert ^{2}+\left\vert
v\right\vert ^{2}\left\vert \xi \right\vert ^{2}-\left( v_{a}\xi ^{a}\right)
^{2}\right) \geq 0
\end{equation}%
since $\left( v_{a}\xi ^{a}\right) ^{2}\leq \left\vert v\right\vert
^{2}\left\vert \xi \right\vert ^{2}$. Therefore, under such a deformation,
the $3$-form $\tilde{\varphi}$ is always a positive $3$-form, and thus
indeed defines a $G_{2}$-structure. The deformation is also invertible.
Suppose we want to use a vector $\tilde{v}$ to get from $\tilde{\varphi}$
back to $\varphi .$ So we have 
\begin{equation}
\varphi =\tilde{\varphi}+\tilde{v}^{e}\tilde{\psi}_{bcde}^{\ \ \ \ \ \ }
\label{g2tilddeform}
\end{equation}%
Then, from (\ref{g2deform}), we obtain 
\begin{equation*}
\tilde{v}^{e}\tilde{\psi}_{bcde}^{\ \ \ \ \ \ }=-v^{e}\psi _{bcde}^{\ \ \ \
\ \ }
\end{equation*}%
Now we multiply both sides by $\tilde{\psi}^{\tilde{b}\tilde{c}\tilde{d}%
\tilde{a}}$. For the left hand side we obtain, using standard contraction
identities%
\begin{equation*}
\tilde{v}^{e}\tilde{\psi}_{bcde}^{\ \ \ \ \ \ }\tilde{\psi}^{\tilde{b}\tilde{%
c}\tilde{d}\tilde{a}}=24\tilde{v}^{a}
\end{equation*}%
For the right hand side, we use the expression for $\tilde{\psi}^{\tilde{b}%
\tilde{c}\tilde{d}\tilde{a}}$ (\ref{psiraised}) 
\begin{eqnarray*}
-v^{e}\psi _{bcde}^{\ \ \ \ \ \ }\tilde{\psi}^{\tilde{b}\tilde{c}\tilde{d}%
\tilde{a}} &=&-\left( 1+M\right) ^{-\frac{2}{3}}v^{e}\psi _{bcde}^{\ \ \ \ \
\ }\left( \psi ^{bcda}+\ast \chi ^{bcda}\right) \\
&=&-\left( 1+M\right) ^{-\frac{2}{3}}v^{e}\psi _{bcde}^{\ \ \ \ \ \ }\left(
\psi ^{bcda}+4v^{[b}\varphi ^{cda]}\right) \\
&=&-24\left( 1+M\right) ^{-\frac{2}{3}}v^{a}
\end{eqnarray*}%
Thus we have the following lemma. Here we decompose $\nabla v$ in terms of
representations of $G_{2}$ as 
\begin{equation}
\nabla _{a}v_{b}=v_{1}g_{ab}+v_{7}^{c}\varphi _{cab}+\left( v_{14}\right)
_{ab}+\left( v_{27}\right) _{ab}
\end{equation}%
where $v_{14}\in \Lambda _{14}^{2}$ and $v_{27}$ is traceless symmetric.

\begin{lemma}
\label{vtildelem}Suppose we have a deformation of a $G_{2}$-structure $%
\varphi $ given by 
\begin{equation*}
\varphi \longrightarrow \tilde{\varphi}=\tilde{\varphi}+v^{e}\psi _{bcde}^{\
\ \ \ \ \ }
\end{equation*}%
Then conversely, the deformation of the new $G_{2}$-structure $\tilde{\varphi%
}$ given by 
\begin{equation*}
\tilde{\varphi}\longrightarrow \tilde{\varphi}+\tilde{v}^{e}\tilde{\psi}%
_{bcde}^{\ \ \ \ \ \ }
\end{equation*}%
results in the original $G_{2}$-structure $\varphi $ if and only if 
\begin{equation*}
\tilde{v}^{a}=-\left( 1+M\right) ^{-\frac{2}{3}}v^{a}
\end{equation*}%
Moreover, $\tilde{v}$ has the following properties:

\begin{enumerate}
\item Denote the norm squared of $\tilde{v}$ with respect to the deformed
metric $\tilde{g}$ by $\tilde{M}$. Then, $\tilde{M}$ is given by 
\begin{equation}
\tilde{M}=\left\vert \tilde{v}\right\vert _{\tilde{g}}^{2}=\tilde{v}^{a}%
\tilde{v}^{b}\tilde{g}_{ab}=M\left( 1+M\right) ^{-2}  \label{mtilde}
\end{equation}

\item The covariant derivative $\tilde{\nabla}$ of $\tilde{v}$ with respect
to the deformed metric $\tilde{g}$ is given by 
\begin{eqnarray}
\tilde{\nabla}_{a}\tilde{v}_{c} &=&\frac{2}{3}\left( 1+M\right) ^{-\frac{7}{3%
}}v_{a}v_{c}\left( \left( 5+2M\right) v_{1}-\left( v_{27}\right)
_{mn}v^{m}v^{n}\right) -\left( 1+M\right) ^{-\frac{1}{3}}\left(
v_{27}\right) _{ac}  \label{deltilvtil} \\
&&-\frac{1}{3}\left( 1+M\right) ^{-\frac{4}{3}}g_{ac}\left( \left(
3+4M\right) v_{1}+\left( v_{27}\right) _{mn}v^{m}v^{n}\right) -\left(
1+M\right) ^{-\frac{4}{3}}\left( v_{14}\right) _{ac}  \notag \\
&&+\frac{1}{3}\left( 1+M\right) ^{-\frac{7}{3}}\left( 3v^{b}\left(
v_{27}\right) _{ba}v_{c}+\left( 1+3M\right) v^{b}\left( v_{27}\right)
_{bc}v_{a}\right)  \notag \\
&&+\frac{2}{3}\left( 1+M\right) ^{-\frac{7}{3}}\left( 3\varphi _{abd}^{\ \
}v^{b}\left( v_{7}\right) ^{d}v_{c}-\varphi _{cbd}^{\ \ }v^{b}\left(
v_{7}\right) ^{d}v_{a}\right)  \notag \\
&&+\frac{2}{3}\left( 1+M\right) ^{-\frac{7}{3}}\left( v_{a}v^{b}\left(
v_{14}\right) _{bc}-3v_{c}v^{b}\left( v_{14}\right) _{ba}\right) -\left(
1+M\right) ^{-\frac{4}{3}}\left( v_{7}\right) ^{b}\varphi _{bac}  \notag
\end{eqnarray}

\item Moreover, $\tilde{\nabla}\tilde{v}=0$ if and only if $\nabla v=0$.
\end{enumerate}

\begin{proof}
The calculation of $\tilde{M}$ is immediate. For the second part, we apply
Lemma \ref{lemdelconn} to calculate $\tilde{\nabla}_{a}\tilde{v}_{c}$.
Finally for the third part, it is trivial that $\tilde{\nabla}\tilde{v}=0$
if $\nabla v=0$. For the \textquotedblleft only if\textquotedblright\
statement, we can invert (\ref{deltilvtil}) to get $\nabla v$ in terms of $%
\tilde{\nabla}\tilde{v}$, in which case in becomes clear that $\nabla v=0$
if $\tilde{\nabla}\tilde{v}=0$.
\end{proof}
\end{lemma}

Now let us use the expression for deformed torsion (\ref{Tanlow2}) to write
it down in terms of $v$. First, we have 
\begin{equation*}
\nabla _{d}s_{pq}=2g_{pq}v_{m}\left( \nabla _{d}v^{m}\right) -\left( \nabla
_{d}v_{p}\right) v_{q}-\left( \nabla _{d}v_{q}\right) v_{p}
\end{equation*}%
and thus, 
\begin{eqnarray*}
\delta _{n}^{c}\nabla _{b}s_{ad}-\frac{1}{9}\delta _{a}^{c}\tilde{g}_{bn}%
\tilde{g}^{\tilde{p}\tilde{q}}\nabla _{d}s_{pq}\newline
&=&\nabla _{e}s_{pq}\left( \delta _{n}^{c}\delta _{b}^{e}\delta
_{a}^{p}\delta _{d}^{q}-\frac{1}{9}\delta _{a}^{c}\delta _{d}^{e}\tilde{g}%
_{bn}\tilde{g}^{\tilde{p}\tilde{q}}\right) \\
&=&\left( 1+M\right) ^{-1}\nabla _{e}s_{pq}\left( \left( 1+M\right) \delta
_{n}^{c}\delta _{b}^{e}\delta _{a}^{p}\delta _{d}^{q}\right. \\
&&\left. -\frac{1}{9}\delta _{a}^{c}\delta _{d}^{e}\left( g_{bn}\left(
1+M\right) -v_{b}v_{n}\right) \left( g^{pq}+v^{p}v^{q}\right) \right)
\end{eqnarray*}%
So, overall, we have 
\begin{eqnarray}
\tilde{T}_{an} &=&\frac{1}{24}\left( 1+M\right) ^{-\frac{4}{3}}\left( \left(
24T_{a}^{\ m}+T_{a}^{\ e}\psi _{ebcd}\ast \chi ^{mbcd}+\psi ^{mbcd}\nabla
_{a}\chi _{bcd}\right. \right.  \label{tantildefull} \\
&&+\left. \nabla _{a}\chi _{bcd}^{\ }\ast \chi ^{mbcd}\right) s_{mn}-3\left(
1+M\right) ^{-1}\left( \varphi _{c}^{\ bd}-\chi _{c}^{\ \
bd}+v^{b}v_{m}\varphi _{c\ \ \ }^{\ \ dm}-v^{d}v_{m}\varphi _{c\ \ \ \ }^{\
\ bm}\right) \times  \notag \\
&&\left. \times \nabla _{e}s_{pq}\left( \left( 1+M\right) \delta
_{n}^{c}\delta _{b}^{e}\delta _{a}^{p}\delta _{d}^{q}-\frac{1}{9}\delta
_{a}^{c}\delta _{d}^{e}\left( g_{bn}\left( 1+M\right) -v_{b}v_{n}\right)
\left( g^{pq}+v^{p}v^{q}\right) \right) \right) .  \notag
\end{eqnarray}%
It makes sense to expand $\nabla v$ also in terms of $G_{2}$-representations:%
\begin{equation}
\nabla _{a}v_{b}=v_{1}g_{ab}+v_{7}^{c}\varphi _{cab}+\left( v_{14}\right)
_{ab}+\left( v_{27}\right) _{ab}  \label{delvdecom}
\end{equation}%
where $v_{14}\in \Lambda _{14}^{2}$ and $v_{27}$ is traceless symmetric.
Together with the similar expansion of $T_{an}$ (\ref{torsioncomps}), after
some manipulations, we obtain:

\begin{theorem}
\label{ThmvTors}Given a $G_{2}$-structure $\varphi $ with full torsion
tensor $T_{ab}$, a deformation of $\varphi $ which lies $\Lambda _{7}^{3}$
given by $\varphi \longrightarrow \varphi +v_{e}\psi _{bcd}^{\ \ \ \ \ \ e}$
results in a new $G_{2}$-structure $\tilde{\varphi}$ with torsion tensor $%
\tilde{T}_{an}$ given by%
\begin{eqnarray}
\tilde{T}_{an} &=&\left( 1+M\right) ^{-\frac{4}{3}}\left( v_{1}\left(
v_{a}v_{n}-\left( 1+M\right) g_{an}\right) -\frac{4}{3}\left( 1+M\right)
v_{1}\varphi _{anm}v^{m}\right.  \label{tantildeexp} \\
&&-\left( 1+\frac{4}{3}M\right) \varphi _{anm}v_{7}^{m}-\frac{1}{3}\psi
_{anmp}v^{m}v_{7}^{p}+\frac{5}{3}v_{a}\varphi _{nmp\ }v^{m}v_{7}^{p}+\frac{4%
}{3}v_{n}\varphi _{amp}v^{m}v_{7}^{p}  \notag \\
&&+\frac{1}{3}v_{7}^{m}v_{m}\varphi _{\ \ an}^{p}v_{p}+\frac{1}{3}%
v_{n}\left( v_{7}\right) _{a}+\frac{8}{3}v_{a}\left( v_{7}\right)
_{n}-\left( 1+M\right) \left( v_{14}\right) _{an}  \notag \\
&&-2v_{m}\left( v_{14}\right) _{\ \ [a}^{m}v_{n]}^{\ }+\frac{1}{3}\varphi
_{anm}v_{14}^{mp}v_{p}+\frac{1}{3}\psi _{anmp}^{\ \ \ \
}v_{q}v^{m}v_{14}^{pq}-\left( 1+M\right) \left( v_{27}\right) _{an}  \notag
\\
&&+v_{m}\left( v_{27}\right) _{\ \ a}^{m}v_{n}-\left( 1+M\right) \varphi _{\
\ \ a}^{mp}\left( v_{27}\right) _{pn}v_{m}-\frac{1}{3}\varphi
_{anm}v_{27}^{mp}v_{p}  \notag \\
&&\left. +\frac{1}{3}\psi _{anmp}^{\ \ \ \
}v^{m}v_{27}^{pq}v_{q}+v_{a}\varphi _{nmp}v^{m}v_{27}^{pq}v_{q}-\frac{1}{3}%
\varphi _{an}^{\ \ \ m}v_{m}v_{27}^{pq}v_{p}v_{q}\right)  \notag \\
&&+\left( 1+M\right) ^{-\frac{1}{3}}\left( \tau _{1}g_{an}+\tau _{1}\varphi
_{\ \ an}^{m}v_{m}+\varphi _{anm}\tau _{7}^{m}+v_{a}\left( \tau _{7}\right)
_{n}-g_{an}\tau _{7}^{m}v_{m}\right.  \notag \\
&&\left. +\psi _{anmp}\tau _{7}^{m}v^{p}+\left( \tau _{14}\right)
_{an}-\varphi _{nmp}v^{m}\left( \tau _{14}\right) _{\ \ a}^{p}+\left( \tau
_{27}\right) _{an}+\varphi _{nmp}v^{m}\left( \tau _{27}\right) _{\ \
a}^{p}\right)  \notag
\end{eqnarray}
\end{theorem}

From this we can also extract the individual components of $\tilde{T}_{an}$
in the representations of $G_{2}$. So first we have the component of $\tilde{%
T}_{an}$ in $\tilde{W}_{1}$: 
\begin{eqnarray}
\tilde{\tau}_{1} &=&\frac{1}{7}\tilde{T}_{ab}\tilde{g}^{\tilde{a}\tilde{b}%
}=\left( 1+M\right) ^{-\frac{1}{3}}\left( \tilde{T}_{ab}g^{ab}+v^{a}v^{b}%
\tilde{T}_{ab}\right)  \notag \\
&=&\left( 1+M\right) ^{-\frac{2}{3}}\left( \left( 1+\frac{1}{7}M\right) \tau
_{1}-v_{1}-\frac{6}{7}\left( \tau _{7}\right) ^{a}v_{a}+\frac{3}{7}\left(
v_{7}\right) ^{a}v_{a}+\frac{1}{7}\left( \tau _{27}\right)
_{ab}v^{a}v^{b}\right)  \label{tau1tild}
\end{eqnarray}

The $7$-dimensional component is given by 
\begin{equation}
\left( \tilde{\tau}_{7}\right) _{c}=\frac{1}{6}\tilde{T}_{ab}\tilde{\varphi}%
_{\ \ \ \ c}^{\tilde{a}\tilde{b}}=\frac{1}{6}\left( 1+M\right) ^{-\frac{2}{3}%
}\tilde{T}_{ab}\left( \varphi _{\ \ \ c}^{ab}-v^{m}\psi _{\ \ \ \
cm}^{ab}+v^{a}v_{m}\varphi _{\ \ \ \ c}^{bm}-v^{b}v_{m}\varphi _{\ \ \ \ \
c}^{am}\right)  \label{tau7tilda}
\end{equation}%
where we have used (\ref{phitildpi7raised2}). Now using the expression for $%
\tilde{T}_{ab}$ (\ref{tantildeexp}), after some manipulations, we obtain%
\begin{eqnarray}
\left( \tilde{\tau}_{7}\right) _{c} &=&\left( \tau _{7}\right) _{c}-\frac{1}{%
6}\varphi _{c}^{\ \ ab}\left( \tau _{7}\right) _{a}v_{b}-\frac{1}{6}%
v^{a}\left( \tau _{27}\right) _{ac}-\frac{1}{6}v^{a}\left( \tau _{14}\right)
_{ac}  \label{tau7tild} \\
&&+\frac{v_{c}}{6\left( 1+M\right) }\left( \left( \tau _{27}\right)
_{ab}v^{a}v^{b}+6\tau _{1}-6\left( \tau _{7}\right) _{a}v^{a}-8v_{1}+3\left(
v_{7}\right) _{a}v^{a}\right)  \notag \\
&&-\frac{1}{6\left( 1+M\right) }\left( 3\left( M+2\right) \left(
v_{7}\right) _{c}+v^{a}\left( v_{27}\right) _{ac}+\varphi _{ca}^{\ \ \
b}v^{a}\left( v_{27}\right) _{bd}v^{d}+3\varphi _{cab}v^{a}\left(
v_{7}\right) ^{b}\right)  \notag
\end{eqnarray}%
Let us now find the $\tilde{W}_{14}$ component. We have $\tilde{T}_{[an]}\in
\Lambda ^{2}$, so 
\begin{equation}
\pi _{14}\left( \tilde{T}_{[an]}\right) =\frac{2}{3}\tilde{T}_{[an]}-\frac{1%
}{6}\tilde{T}_{mp}\tilde{\psi}_{\ \ \ \ \ an}^{\tilde{m}\tilde{p}}
\label{pi14ttild1}
\end{equation}%
The skew-symmetric part of (\ref{tantildeexp}) is given by:

\begin{eqnarray}
\tilde{T}_{[an]} &=&\left( 1+M\right) ^{-\frac{4}{3}}\left( -\frac{4}{3}%
\left( 1+M\right) v_{1}^{\ \ \ \ }\varphi _{anm}v^{m}-\left( 1+\frac{4}{3}%
M\right) \varphi _{anm}v_{7}^{m}\right.  \label{tTildeasym} \\
&&-\frac{1}{3}\psi _{anmp}v^{m}v_{7}^{p}+v_{[a}\varphi _{n]mp\
}v^{m}v_{7}^{p}+\frac{1}{3}v_{7}^{m}v_{m}\varphi _{\ an}^{p}v_{p}+\frac{7}{3}%
v_{[a}\left( v_{7}\right) _{n]}  \notag \\
&&-\left( 1+M\right) \left( v_{14}\right) _{an}-2v_{m}\left( v_{14}\right)
_{\ \ [a}^{m}v_{n]}^{\ }+\frac{1}{3}\varphi _{anm}v_{14}^{mp}v_{p}+\frac{1}{3%
}\psi _{anmp}^{\ }v_{q}v^{m}v_{14}^{pq}  \notag \\
&&+v_{m}\left( v_{27}\right) _{\ \ [a}^{m}v_{n]}-\left( 1+M\right) \varphi
_{\ \ \ [a}^{mp}\left( v_{27}\right) _{n]p}v_{m}-\frac{1}{3}\varphi
_{anm}v_{27}^{mp}v_{p}  \notag \\
&&\left. +\frac{1}{3}\psi _{anmp}^{\ \ \ \
}v^{m}v_{27}^{pq}v_{q}+v_{[a}\varphi _{n]mp}v^{m}v_{27}^{pq}v_{q}-\frac{1}{3}%
\varphi _{an}^{\ \ \ \ m}v_{m}v_{27}^{pq}v_{p}v_{q}\right)  \notag \\
&&+\left( 1+M\right) ^{-\frac{1}{3}}\left( \tau _{1}\varphi _{\ \
an}^{m}v_{m}+\varphi _{anm}\tau _{7}^{m}+v_{[a}\left( \tau _{7}\right)
_{n]}+\psi _{anmp}\tau _{7}^{m}v^{p}\right.  \notag \\
&&\left. +\left( \tau _{14}\right) _{an}+\varphi _{mp[a}v^{m}\left( \tau
_{14}\right) _{\ \ n]}^{p}-\varphi _{mp[a}v^{m}\left( \tau _{27}\right) _{\
\ n]}^{p}\right)  \notag
\end{eqnarray}%
Now note that 
\begin{equation*}
\tilde{\psi}_{\ \ \ an}^{\tilde{m}\tilde{p}}=\tilde{\psi}^{\tilde{m}\tilde{p}%
\tilde{q}\tilde{r}}\tilde{g}_{qa}\tilde{g}_{nr}=\left( 1+M\right)
^{-2}\left( \psi ^{mpqr}+4v^{[m}\varphi ^{pqr]}\right) \left( g_{aq}\left(
1+M\right) -v_{a}v_{q}\right) \left( g_{nr}\left( 1+M\right)
-v_{n}v_{r}\right)
\end{equation*}%
Hence the $14$-dimensional component is 
\begin{eqnarray}
\left( \tilde{\tau}_{14}\right) _{an} &=&\left( 1+M\right) ^{-\frac{4}{3}%
}\left( \frac{10}{3}\left( v_{7}\right) _{[a}v_{n]}+\frac{4}{3}v_{[a}\varphi
_{\ \ \ \ n]}^{mp}v_{m}\left( v_{7}\right) _{p}-\left( \frac{5}{6}+\frac{1}{2%
}M\right) \psi _{\ \ \ an}^{mp}v_{m}\left( v_{7}\right) _{p}\right.
\label{tTilde14} \\
&&+\frac{1}{3}\left( v_{7}\right) _{m}v^{m}v_{p}\varphi _{\ \ an}^{p}-\frac{1%
}{3}M\left( v_{7}\right) _{m}\varphi _{\ \ an}^{m}-\left( 1+M\right) \left(
v_{14}\right) _{an}-2v_{m}\left( v_{14}\right) _{\ \ \ [a}^{m}v_{n]}  \notag
\\
&&+\frac{1}{3}\varphi _{\ \ an}^{m}v_{p}\left( v_{14}\right) _{\ \ m}^{p}+%
\frac{1}{3}\psi _{\ \ \ an}^{mp}v_{m}\left( v_{14}\right) _{pq}v^{q}-\frac{1%
}{3}\varphi _{\ \ \ an}^{m}v_{m}\left( v_{27}\right) _{pq}v^{p}v^{q}  \notag
\\
&&+\left( M+1\right) \varphi _{\ \ \ [a}^{mp}\left( v_{27}\right) _{n]p}^{\
}v_{m}+\frac{1}{6}\left( M-1\right) \varphi _{\ \ an}^{m}\left(
v_{27}\right) _{\ \ m}^{p}v_{p}  \notag \\
&&\left. +\frac{2}{3}v_{m}\left( v_{27}\right) _{\ \ \ [a}^{m}v_{n]}-\frac{4%
}{3}\varphi _{\ \ \ \ [a}^{mp}v_{n]}v_{m}\left( v_{27}\right) _{pq}v^{q}+%
\frac{1}{3}\psi _{\ \text{\ }an}^{mp}v_{m}\left( v_{27}\right)
_{pq}v^{q}\right)  \notag \\
&&+\left( 1+M\right) ^{-\frac{1}{3}}\left( -\frac{1}{6}M\varphi _{\ \
an}^{m}\left( \tau _{7}\right) _{m}+\frac{1}{6}\psi _{\ \ \ an}^{mp}\left(
\tau _{7}\right) _{m}v_{p}-\frac{1}{3}\varphi _{\ \ \ \ [a}^{mp}v_{n]}\left(
\tau _{7}\right) _{m}v_{p}+\frac{2}{3}v_{[a}\left( \tau _{7}\right)
_{n]}\right.  \notag \\
&&+\frac{1}{6}\varphi _{\ \ \ an}^{m}v_{m}v_{p}\left( \tau _{7}\right)
^{p}+\left( \tau _{14}\right) _{an}+\frac{1}{6}\psi _{\ \ \ \ an}^{mp}\left(
\tau _{14}\right) _{\ \ m}^{q}v_{p}v_{q}-\frac{1}{3}\varphi _{\ \
an}^{m}v_{p}\left( \tau _{14}\right) _{\ \ m}^{p}  \notag \\
&&\left. -\varphi _{\ \ \ [a}^{mp}\left( \tau _{27}\right) _{n]p}v_{m}+\frac{%
1}{6}\varphi _{\ \ an}^{m}\left( \tau _{27}\right) _{\ \ m}^{p}v_{p}+\frac{1%
}{6}\psi _{\ \ \ \ an}^{mp}\left( \tau _{27}\right) _{\ \ \
m}^{q}v_{p}v_{q}\right)  \notag
\end{eqnarray}

Finally, the component in $\tilde{W}_{27}$ is now given by 
\begin{equation*}
\left( \tilde{\tau}_{27}\right) _{an}=\tilde{T}_{(an)}-\tau _{1}\tilde{g}%
_{an}
\end{equation*}%
where $\tilde{T}_{(an)}$ is the symmetric part of (\ref{tantildeexp}):%
\begin{eqnarray}
\tilde{T}_{(an)} &=&\left( 1+M\right) ^{-\frac{4}{3}}\left( v_{1}\left(
v_{a}v_{n}-\left( 1+M\right) g_{an}\right) +3v_{(a}\varphi _{n)mp\
}v^{m}v_{7}^{p}\right.  \label{tTildesym} \\
&&+3v_{(a}\left( v_{7}\right) _{n)}-\left( 1+M\right) \left( v_{27}\right)
_{an}+v_{m}\left( v_{27}\right) _{\ \ (a}^{m}v_{n)}  \notag \\
&&\left. -\left( 1+M\right) \varphi _{\ \ \ (a}^{mp}\left( v_{27}\right)
_{n)p}v_{m}+v_{(a}\varphi _{n)mp}v^{m}v_{27}^{pq}v_{q}\right)  \notag \\
&&+\left( 1+M\right) ^{-\frac{1}{3}}\left( \tau _{1}g_{an}+v_{(a}\left( \tau
_{7}\right) _{n)}-g_{an}\tau _{7}^{m}v_{m}\right.  \notag \\
&&\left. -\varphi _{mp(a}v^{m}\left( \tau _{14}\right) _{\ \ n)}^{p}+\left(
\tau _{27}\right) _{an}+\varphi _{mp(a}v^{m}\left( \tau _{27}\right) _{\ \
n)}^{p}\right)  \notag
\end{eqnarray}%
and 
\begin{equation*}
\tilde{\tau}_{1}\tilde{g}_{an}=\left( 1+M\right) ^{-\frac{4}{3}}\left(
\left( 1+M\right) g_{an}-v_{a}v_{n}\right) \left( \left( 1+\frac{1}{7}%
M\right) \tau _{1}-v_{1}-\frac{6}{7}\tau _{7}^{m}v_{m}+\frac{3}{7}%
v_{7}^{m}v_{m}-\frac{1}{7}\tau _{27}^{mp}v_{m}v_{p}\right)
\end{equation*}%
Thus overall, we have 
\begin{eqnarray}
\left( \tilde{\tau}_{27}\right) _{an} &=&\left( 1+M\right) ^{-\frac{4}{3}%
}\left( -\frac{3}{7}\left( \left( 1+M\right) g_{an}-v_{a}v_{n}\right)
v_{7}^{m}v_{m}+3v_{(a}\varphi _{n)mp\ }v^{m}v_{7}^{p}\right.
\label{tTilde27} \\
&&+3v_{(a}\left( v_{7}\right) _{n)}-\left( 1+M\right) \left( v_{27}\right)
_{an}+v_{m}\left( v_{27}\right) _{\ \ (a}^{m}v_{n)}-\left( 1+M\right)
\varphi _{\ \ \ (a}^{mp}\left( v_{27}\right) _{n)p}v_{m}  \notag \\
&&\left. +v_{(a}\varphi _{n)mp}v^{m}v_{27}^{pq}v_{q}+\left( 1+\frac{1}{7}%
M\right) \tau _{1}v_{a}v_{n}-\frac{6}{7}\tau _{7}^{m}v_{m}v_{a}v_{n}+\frac{1%
}{7}\tau _{27}^{mp}v_{m}v_{p}v_{a}v_{n}\right)  \notag \\
&&+\left( 1+M\right) ^{-\frac{1}{3}}\left( -\frac{1}{7}M\tau
_{1}g_{an}+v_{(a}\left( \tau _{7}\right) _{n)}-\frac{1}{7}g_{an}\tau
_{7}^{m}v_{m}\right.  \notag \\
&&\left. -\varphi _{mp(a}v^{m}\left( \tau _{14}\right) _{\ \ n)}^{p}+\left(
\tau _{27}\right) _{an}+\varphi _{mp(a}v^{m}\left( \tau _{27}\right) _{\ \
n)}^{p}-\frac{1}{7}\tau _{27}^{mp}v_{m}v_{p}g_{an}\right)  \notag
\end{eqnarray}%
The expressions (\ref{tau1tild}), (\ref{tau7tild}), (\ref{tTilde14}) and (%
\ref{tTilde27}) give us the components of the new torsion $\tilde{T}$ in $%
\tilde{W}_{1}$, $\tilde{W}_{7}$, $\tilde{W}_{14}$ and $\tilde{W}_{27}$,
respectively. As we can see these expressions are quite complicated, so for
a generic deformation vector $v$, in general we would obtain

\begin{theorem}
\label{ThmFullDeform}Given a $G_{2}$-structure $\varphi $ with full torsion
tensor $T_{ab}$, a deformation of $\varphi $ which lies $\Lambda _{7}^{3}$
given by $\varphi \longrightarrow \varphi +v_{e}\psi _{bcd}^{\ \ \ \ \ \ e}$
results in a new $G_{2}$-structure $\tilde{\varphi}$ with a torsion tensor $%
\tilde{T}_{an}$ if only if the components $v_{1}$,$v_{7}$, $v_{14}$ and $%
v_{27}$ of $\nabla v$ satisfy the following equations 
\begin{subequations}%
\label{tt0allsol} 
\begin{eqnarray}
v_{1} &=&\tau _{1}-\frac{3}{7}\left( \tau _{7}\right) _{a}v^{a}-\frac{1}{7}%
\frac{\left( 7+3M\right) }{\left( 1+M\right) ^{\frac{1}{3}}}\tilde{\tau}_{1}-%
\frac{3}{7}\left( \tilde{\tau}_{7}\right) _{a}v^{a}+\frac{1}{14}\left(
1+M\right) ^{\frac{1}{3}}\left( \tilde{\tau}_{27}\right) _{ab}v^{a}v^{b}
\label{v1sol} \\
\left( v_{7}\right) ^{c} &=&\left( \tau _{7}\right) ^{c}-\frac{1}{3}\tau
_{1}v^{c}+\frac{1}{3}\varphi _{\ \ ab}^{c}\left( \tau _{7}\right) ^{a}v^{b}-%
\frac{1}{6}\left( \tau _{14}\right) _{a}^{\ \ c}v^{a}-\frac{1}{3}\left( \tau
_{27}\right) _{a}^{\ \ c}v^{a}+\frac{4}{3}\frac{\tilde{\tau}_{1}}{\left(
1+M\right) ^{\frac{1}{3}}}v^{c}  \label{v7sol} \\
&&-\left( \tilde{\tau}_{7}\right) ^{c}-\frac{1}{2}\varphi _{\ \
ab}^{c}\left( \tilde{\tau}_{7}\right) ^{a}v^{b}+\frac{1}{6}\left( 1+M\right)
^{\frac{1}{3}}\left( \tilde{\tau}_{27}\right) _{a}^{\ \ c}v^{a}  \notag \\
\left( v_{14}\right) _{ab} &=&\frac{1}{\left( M+9\right) }\left( \frac{4}{3}%
\left( M-27\right) \left( \tau _{7}\right) _{[a}v_{b]}-\frac{1}{3}\left(
M-27\right) \psi _{\ \ \ \ ab}^{mn}\left( \tau _{7}\right)
_{m}v_{n}-4M\varphi _{\ \ ab}^{m}\left( \tau _{7}\right) _{m}\right.
\label{v14sol} \\
&&+4\left( \tau _{7}\right) ^{m}v_{m}\varphi _{\ \ ab}^{n}v_{n}-24\varphi
_{\ \ \ \ \ [a}^{mn}v_{b]}\left( \tau _{7}\right) _{m}v_{n}+\frac{1}{2}%
\left( M+2\right) \left( M+9\right) \left( \tau _{14}\right) _{ab}  \notag \\
&&+\frac{1}{2}\left( M+9\right) \varphi _{a}^{\ \ mn}\varphi _{b}^{\ \ \
pq}v_{m}v_{p}\left( \tau _{14}\right) _{nq}+\left( M-7\right) v_{m}\left(
\tau _{14}\right) _{\ \ [a}^{m}v_{b]}  \notag \\
&&+8\varphi _{\ \ \ \ \ [a}^{mn}v_{b]}v_{n}v_{p}\left( \tau _{14}\right) _{\
\ m}^{p}-\frac{4}{3}Mv_{m}\left( \tau _{14}\right) _{\ \ n}^{m}\varphi _{\ \
ab}^{n}+4\psi _{\ \ \ \ \ ab}^{mn}v_{p}\left( \tilde{\tau}_{14}\right) _{\ \
m}^{p}v_{n}  \notag \\
&&+16v_{m}\left( \tau _{27}\right) _{\ \ [a}^{m}v_{b]}-\frac{4}{3}%
v_{m}\varphi _{\ \ ab}^{m}\left( \tau _{27}\right) _{np}v^{n}v^{p}+\frac{4}{3%
}Mv_{m}\left( \tau _{27}\right) _{\ \ n}^{m}\varphi _{\ \ ab}^{n}-4\psi _{\
\ \ \ \ ab}^{mn}v_{p}\left( \tau _{27}\right) _{\ \ m}^{p}v_{n}  \notag \\
&&+8\varphi _{\ \ \ \ \ [a}^{mn}v_{b]}v_{n}v_{p}\left( \tau _{27}\right) _{\
\ m}^{p}+\left( \frac{1}{6}\left( M+17\right) \left( 1+M\right) ^{\frac{1}{3}%
}\left( \tilde{\tau}_{27}\right) _{mn}v^{m}v^{n}-4\left( \tilde{\tau}%
_{7}\right) _{m}v^{m}\right) v_{p}\varphi _{\ \ ab}^{p}  \notag \\
&&+24\varphi _{\ \ \ \ \ [a}^{mn}v_{b]}\left( \tilde{\tau}_{7}\right)
_{m}v_{n}-2\left( M-15\right) \left( \tilde{\tau}_{7}\right) _{[a}v_{b]}+%
\frac{1}{2}\left( M-15\right) \psi _{\ \ \ \ ab}^{mn}\left( \tilde{\tau}%
_{7}\right) _{m}v_{n}  \notag \\
&&+4M\varphi _{\ \ ab}^{m}\left( \tilde{\tau}_{7}\right) _{m}-\left(
1+M\right) ^{\frac{1}{3}}\left( M+9\right) \left( \left( \tilde{\tau}%
_{14}\right) _{ab}+v^{m}\varphi _{mn[a}\left( \tilde{\tau}_{27}\right) _{\
b]}^{n}\right)  \notag \\
&&+16\left( 1+M\right) ^{-\frac{2}{3}}v_{m}\left( \tilde{\tau}_{14}\right)
_{\ \ [a}^{m}v_{b]}+8\left( 1+M\right) ^{-\frac{2}{3}}\varphi _{\ \ \ \ \
[a}^{mn}v_{b]}v_{n}v_{p}\left( \tilde{\tau}_{14}\right) _{\ \ m}^{p}  \notag
\\
&&+\left( M-3\right) \left( 1+M\right) ^{-\frac{2}{3}}v_{m}\left( \tilde{\tau%
}_{14}\right) _{\ \ n}^{m}\varphi _{\ \ ab}^{n}-4\left( 1+M\right) ^{-\frac{2%
}{3}}\psi _{\ \ \ \ \ ab}^{mn}v_{p}\left( \tilde{\tau}_{14}\right) _{\ \
m}^{p}v_{n}  \notag \\
&&-8\left( 1+M\right) ^{\frac{1}{3}}\varphi _{\ \ \ \ \
[a}^{mn}v_{b]}v_{n}v_{p}\left( \tilde{\tau}_{27}\right) _{\ \ m}^{p}+8\left(
1+M\right) ^{\frac{1}{3}}v_{m}\left( \tilde{\tau}_{27}\right) _{\ \
[a}^{m}v_{b]}  \notag \\
&&\left. +\frac{1}{6}\left( 7M-9\right) \left( 1+M\right) ^{\frac{1}{3}%
}v_{m}\left( \tilde{\tau}_{27}\right) _{\ \ n}^{m}\varphi _{\ \
ab}^{n}-4\left( 1+M\right) ^{\frac{1}{3}}\psi _{\ \ \ \ \
ab}^{mn}v_{p}\left( \tilde{\tau}_{27}\right) _{\ \ m}^{p}v_{n}\right)  \notag
\\
\left( v_{27}\right) _{ab} &=&\left( \tau _{27}\right) _{ab}+4\left( \tau
_{7}\right) _{(a}v_{b)}+\left( 4\left( 1+M\right) ^{-\frac{1}{3}}\tilde{\tau}%
_{1}-\frac{1}{2}\left( 1+M\right) ^{-\frac{2}{3}}\left( \tilde{\tau}%
_{27}\right) _{mn}v^{m}v^{n}\right) v_{a}v_{b}  \label{v27sol} \\
&&-\frac{1}{7}\left( 4\left( \tau _{7}\right) _{m}v^{m}-3\left( \tilde{\tau}%
_{7}\right) _{m}v^{m}-4M\left( 1+M\right) ^{-\frac{1}{3}}\tilde{\tau}_{1}-%
\frac{1}{2}\left( 1+M\right) ^{-\frac{1}{3}}\left( \tilde{\tau}_{27}\right)
_{mn}v^{m}v^{n}\right) g_{ab}  \notag \\
&&-3\tilde{\tau}_{7(a}v_{b)}-\varphi _{\ \ \ (a}^{mn}\left( \tilde{\tau}%
_{14}\right) _{b)n}v_{m}-\frac{1}{2}\left( 1+M\right) ^{-\frac{2}{3}}\varphi
_{a}^{\ \ mn}\varphi _{b}^{\ \ pq}\left( \tilde{\tau}_{27}\right)
_{mp}v_{n}v_{q}  \notag \\
&&-\left( 1+M\right) ^{-\frac{2}{3}}\left( \frac{1}{2}\left( 2+M\right)
\left( \tilde{\tau}_{27}\right) _{ab}+v_{m}\left( \tilde{\tau}_{27}\right)
_{\ \ (a}^{m}v_{b)}+v^{m}\varphi _{mn(a}\left( \tilde{\tau}_{27}\right) _{\
\ b)}^{n}\right.  \notag \\
&&\left. +\varphi _{\ \ \ \ \ (a}^{mn}v_{b)}v_{n}v_{p}\left( \tilde{\tau}%
_{27}\right) _{\ \ m}^{p}\right)  \notag
\end{eqnarray}%
\end{subequations}%
and moreover, the following necessary condition is satisfied: 
\begin{equation}
d\left( \left( v_{7}\right) \lrcorner \varphi +v_{14}\right) =0.
\label{dvcond}
\end{equation}
\end{theorem}

\begin{proof}
In order to obtain the equations which the components $v_{1}$, $v_{7}$, $%
v_{14}$ and $v_{27}$ of $\nabla v$ must satisfy, we have to invert the
expressions for $\tilde{\tau}_{1},$ $\tilde{\tau}_{7},$ $\tilde{\tau}_{14}$
and $\tilde{\tau}_{27}$ that are given by (\ref{tau1tild}), (\ref{tau7tilda}%
), (\ref{tTilde14}) and (\ref{tTilde27}), respectively. So we solve for $%
v_{1}$, $\left( v_{7}\right) ^{c}$, $\left( v_{14}\right) _{ab}$ and $\left(
v_{27}\right) _{ab}$ in terms of the original torsion components $\tau _{1},$
$\tau _{7},$ $\tau _{14}$ and $\tau _{27}$ and the new torsion components $%
\tilde{\tau}_{1},$ $\tilde{\tau}_{7},$ $\tilde{\tau}_{14}$ and $\tilde{\tau}%
_{27}$. Therefore pointwise we have the same number of variables as
equations, so generically we should be able to solve it.

For convenience, let us denote the left hand sides of equations (\ref%
{tau1tild}), (\ref{tau7tilda}), (\ref{tTilde14}) and (\ref{tTilde27}), by $%
\hat{\tau}_{1}$,$\hat{\tau}_{7},$ $\hat{\tau}_{14}$ and $\hat{\tau}_{27}$,
respectively. Hence these equations can be rewritten as 
\begin{eqnarray*}
\tilde{\tau}_{1} &=&\hat{\tau}_{1} \\
\tilde{\tau}_{7} &=&\hat{\tau}_{7} \\
\tilde{\tau}_{14} &=&\hat{\tau}_{14} \\
\tilde{\tau}_{27} &=&\hat{\tau}_{27}
\end{eqnarray*}

Let us first look at the $\tilde{\tau}_{1}$ equation. Note that the
expression for $\tilde{\tau}_{1}$ contains the scalars $v_{1}$ and $\left(
v_{7}\right) ^{a}v_{a}$. So in order to find $v_{1}$, we would also need to
find $\left( v_{7}\right) ^{a}v_{a}$. We can get another equation that has $%
\left( v_{7}\right) ^{a}v_{a}$ by constructing the scalar $\left( \tilde{\tau%
}_{7}\right) ^{c}v_{c}.$ However that now also has the scalar $\left(
v_{27}\right) ^{ab}v_{a}v_{b}$. So we would need another equation - this
time from by $\left( \tilde{\tau}_{27}\right) _{ab}v^{a}v^{b}$. Now we can
solve the system 
\begin{subequations}
\label{ttscal0eqs2}
\begin{eqnarray}
\tilde{\tau}_{1} &=&\hat{\tau}_{1}=\left( 1+M\right) ^{-\frac{2}{3}}\left(
-v_{1}+\frac{3}{7}\left( v_{7}\right) ^{a}v_{a}+\left( 1+\frac{1}{7}M\right)
\tau _{1}-\frac{6}{7}\left( \tau _{7}\right) ^{a}v_{a}\right. \\
&&\left. +\frac{1}{7}\left( \tau _{27}\right) ^{ab}v_{a}v_{b}\right)  \notag
\\
\left( \tilde{\tau}_{7}\right) ^{c}v_{c} &=&\left( \hat{\tau}_{7}\right)
^{c}v_{c}=\left( 1+M\right) ^{-1}\left( -\frac{4}{3}Mv_{1}-\left(
v_{7}\right) ^{a}v_{a}-\frac{1}{6}\left( v_{27}\right) ^{ab}v_{a}v_{b}+M\tau
_{1}\right. \\
&&\left. +\left( \tau _{7}\right) ^{a}v_{a}-\frac{1}{6}\left( \tau
_{27}\right) ^{ab}v_{a}v_{b}\right)  \notag \\
\left( \tilde{\tau}_{27}\right) _{ab}v^{a}v^{b} &=&\left( \hat{\tau}%
_{27}\right) _{ab}v^{a}v^{b}=\left( 1+M\right) ^{-\frac{4}{3}}\left( \frac{18%
}{7}M\left( v_{7}\right) ^{a}v_{a}-\left( v_{27}\right) ^{ab}v_{a}v_{b}+%
\frac{6}{7}M^{2}\tau _{1}\right. \\
&&\left. +\frac{6}{7}M\left( \tau _{7}\right) ^{a}v_{a}+\left( 1+\frac{6}{7}%
M\right) \left( \tau _{27}\right) ^{ab}v_{a}v_{b}\right)  \notag
\end{eqnarray}%
We have three equations, and the three variables, $v_{1}$, $\left(
v_{7}\right) ^{a}v_{a}$ and $\left( v_{27}\right) _{ab}v^{a}v^{b}$, which we
treat as being independent. The determinant of this system is proportional
to $\left( M+1\right) $, but $M=\left\vert v\right\vert ^{2}>0,$ so we
always have a solution. Solving, we get the solution (\ref{v1sol}) for $%
v_{1} $ and also solutions for $\left( v_{7}\right) ^{a}v_{a}$ and $\left(
v_{27}\right) _{ab}v^{a}v^{b}$.

Note that we could have also considered $\left( \tilde{\tau}_{27}\right)
_{ab}g^{ab}$. However, since $\tilde{\tau}_{27}$ is traceless with respect
to $\tilde{g}^{ab}=\left( 1+M\right) ^{-\frac{2}{3}}\left(
g^{ab}+v^{a}v^{b}\right) $, 
\end{subequations}
\begin{equation*}
\left( \tilde{\tau}_{27}\right) _{ab}g^{ab}=-\left( \tilde{\tau}_{27}\right)
_{ab}v^{a}v^{b}
\end{equation*}%
so we would get no new independent equation.

Next, we look at the $\tilde{\tau}_{7}$ equation. We now have expressions
for $\left( v_{7}\right) ^{a}v_{a}$ and $\left( v_{27}\right)
^{ab}v_{a}v_{b} $, so we can replace any instances of these scalars by the
solutions of the above scalar equations. Our remaining variables are now $%
\left( v_{7}\right) ^{c},\left( v_{14}\right) _{\ a}^{\ \ \ c}v^{a},\left(
v_{27}\right) _{a}^{\ \ \ c}v^{a},\ \varphi _{\ \ ab}^{c}\left( v_{7}\right)
^{a}v^{b},\varphi _{\ \ ab}^{c}\left( v_{14}\right) ^{eb}v^{a}v_{e}$ and $%
\varphi _{\ \ ab}^{c}\left( v_{27}\right) ^{be}v^{a}v_{e}$. To solve for
these variables, we construct six equations 
\begin{subequations}
\label{tt0vecteqs}
\begin{eqnarray}
&&\left( \tilde{\tau}_{7}\right) _{a}=\left( \hat{\tau}\right) _{a} \\
&&\left( \tilde{\tau}_{14}\right) _{ab}v^{a}=\left( \hat{\tau}_{14}\right)
_{ab}v^{a} \\
&&\left( \tilde{\tau}_{27}\right) _{ab}v^{a}=\left( \hat{\tau}_{27}\right)
_{ab}v^{a} \\
&&\varphi _{\ \ bc}^{a}\left( \tilde{\tau}_{7}\right) ^{b}v^{c}=\varphi _{\
\ bc}^{a}\left( \hat{\tau}_{7}\right) ^{b}v^{c} \\
&&\varphi _{\ \ bc}^{a}\left( \tilde{\tau}_{14}\right) _{\ \
d}^{b}v^{d}v^{c}=\varphi _{\ \ bc}^{a}\left( \hat{\tau}_{14}\right) _{\ \
d}^{b}v^{d}v^{c} \\
&&\varphi _{\ \ bc}^{a}\left( \tilde{\tau}_{27}\right) _{\ \
d}^{b}v^{d}v^{c}=\varphi _{\ \ bc}^{a}\left( \hat{\tau}_{27}\right) _{\ \
d}^{b}v^{d}v^{c}
\end{eqnarray}%
The left hand side of each of these equations is now some function of $v,$ $%
\tau _{1},$ $\tau _{7},$ $\tau _{14}$ and $\tau _{27}$ constructed from the
expressions for $\tilde{\tau}_{1},$ $\tilde{\tau}_{7},$ $\tilde{\tau}_{14}$
and $\tilde{\tau}_{27}$ and with any instances of $v_{1}$,$\left(
v_{7}\right) ^{a}v_{a}$ and $\left( v_{27}\right) ^{ab}v_{a}v_{b}$ replaced
by the solutions of equations (\ref{ttscal0eqs2}). It turns out that we do
not get any new variables, and so we get six equations for six variables.
The determinant of this system is positive, so we can solve this, and in
particular, get the solution for $\left( v_{7}\right) ^{c}$ (\ref{v7sol}).
We also get solutions for the other vectors constructed above.

Now we can look at the last two equation - $\left( \tilde{\tau}_{14}\right)
_{ab}=0$ and $\left( \tilde{\tau}_{27}\right) _{ab}=0$. We now have
solutions for scalars and vectors, so we can substitute them into these
equations. Then, the variables in the first equation are skew-symmetric
quantities, and in the second equation we have symmetric quantities.

In the $\tilde{\tau}_{14}$ equation the quantities are $\left( v_{14}\right)
_{ab}$ and $\varphi _{\ \ \ [a}^{cd}\left( v_{27}\right) _{b]d}v_{c}$, while
in the $\tilde{\tau}_{27}$ equation we have $\left( v_{27}\right) _{ab}$ and 
$\varphi _{\ \ \ (a}^{cd}\left( v_{27}\right) _{b)d}v_{c}.$ Hence we can
construct quantities $\varphi _{\ \ \ [a}^{cd}\left( \tilde{\tau}%
_{27}\right) _{b]d}v_{c}$ and $\varphi _{\ \ \ (a}^{cd}\left( \tilde{\tau}%
_{27}\right) _{b)d}v_{c}$ which give us one extra equation for both
skew-symmetric and symmetric quantities. For the skew-symmetric equations we
get no new variables, thus our equations are 
\end{subequations}
\begin{subequations}
\label{tt0skeweqs}
\begin{eqnarray}
&&\left( \tilde{\tau}_{14}\right) _{ab}=\left( \hat{\tau}_{14}\right) _{ab}
\\
&&\varphi _{\ \ \ [a}^{cd}\left( \tilde{\tau}_{27}\right)
_{b]d}v_{c}=\varphi _{\ \ \ [a}^{cd}\left( \hat{\tau}_{27}\right) _{b]d}v_{c}
\end{eqnarray}%
Here we solve for $\left( v_{14}\right) _{ab}$ and $\varphi _{\ \ \
[a}^{cd}\left( v_{27}\right) _{b]d}v_{c},$ and immediately get the solution (%
\ref{v14sol}). It can be checked that this expression does indeed give a $2$%
-form lying in $\Lambda _{14}^{2}$.

Going back to the symmetric equations, from $\varphi _{\ \ \ (a}^{cd}\left( 
\tilde{\tau}_{27}\right) _{b)d}v_{c}$ we get a new symmetric variable $%
\varphi _{a}^{\ \ cd}\varphi _{b}^{\ \ \ ef}v_{c}v_{e}\left( v_{27}\right)
_{df}$. We then construct the quantity $\varphi _{a}^{\ \ cd}\varphi _{b}^{\
\ \ ef}v_{c}v_{e}\left( \tilde{\tau}_{27}\right) _{df}$ and get no new
variables. Therefore, the symmetric equations are 
\end{subequations}
\begin{subequations}
\label{tt0symeqs}
\begin{eqnarray}
&&\left( \tilde{\tau}_{27}\right) _{ab}=\left( \hat{\tau}_{27}\right) _{ab}
\\
&&\varphi _{\ \ \ (a}^{cd}\left( \tilde{\tau}_{27}\right)
_{b)d}v_{c}=\varphi _{\ \ \ (a}^{cd}\left( \hat{\tau}_{27}\right) _{b)d}v_{c}
\\
&&\varphi _{a}^{\ \ cd}\varphi _{b}^{\ \ \ ef}v_{c}v_{e}\left( \tilde{\tau}%
_{27}\right) _{df}=\varphi _{a}^{\ \ cd}\varphi _{b}^{\ \ \
ef}v_{c}v_{e}\left( \hat{\tau}_{27}\right) _{df}
\end{eqnarray}%
where we solve for $\left( v_{27}\right) _{ab}$, $\varphi _{\ \ \
(a}^{cd}\left( v_{27}\right) _{b)d}v_{c}$ and $\varphi _{a}^{\ \ cd}\varphi
_{b}^{\ \ \ ef}v_{c}v_{e}\left( v_{27}\right) _{df}$. We have three
equations with three variable, and the determinant is again positive, so we
solve it and get the solution (\ref{v27sol}) for $\left( v_{27}\right) _{ab}$%
. Note that it is always traceless, hence indeed always corresponds to the
component in the $27$-dimensional representation.

To get the necessary condition (\ref{dvcond}), first note that 
\end{subequations}
\begin{equation}
dv^{\flat }=2\left( v_{7}\right) \lrcorner \varphi +2v_{14}.  \label{dv1}
\end{equation}%
Therefore, we must have 
\begin{equation*}
d^{2}v^{\flat }=0\text{,}
\end{equation*}%
which gives us (\ref{dvcond}). So far, we have only considered the algebraic
constraints on the components of $\nabla v$, so the differential condition (%
\ref{dvcond}) is not automatically satisfied in general, and must be imposed
separately.
\end{proof}

Note that the condition (\ref{dvcond}) in Theorem \ref{ThmFullDeform}
involves second derivatives of $v$ - in particular, derivatives of the $%
\mathbf{7}$ and the $\mathbf{14}$ components of $\nabla v$, that is, $v_{7}$
and $v_{14}$. However, from the equations (\ref{v7sol}) and (\ref{v14sol}), $%
v_{7}$ and $v_{14}$ are expressed in terms of $v$ and the torsion
components. However, derivatives of $v$ can be reduced again to expressions
just involving $v$ and the torsion components, using all of the equations (%
\ref{v1sol}) to (\ref{v27sol}). So overall, (\ref{dvcond}) gives a
relationship between $v$, the torsion components and the derivatives of the
torsion components. Moreover, we can also apply the conditions on the
derivatives of torsion components from Proposition \ref{PropTorsConds} in
order to relate some of the torsion derivatives to the torsion components
themselves. In the general case, the resulting expressions are extremely
long, and not very helpful, so we will consider individual torsion classes
in order to gain more insight.

The simplest case is when the original torsion vanishes.

\begin{corollary}
\label{corrt0t0}Suppose the $3$-form $\varphi $ defines a torsion-free $%
G_{2}\,$-structure, then a deformation of $\varphi $ which lies in $\Lambda
_{7}^{3}$ and is given by $\varphi \longrightarrow \varphi +v^{e}\psi
_{bcde}^{\ \ \ \ \ \ }$ results in a new torsion-free $G_{2}$-structure $%
\tilde{\varphi}$ if and only if 
\begin{equation*}
\nabla v=0
\end{equation*}
\end{corollary}

\begin{proof}
We get this immediately by setting $\tau _{1}=\tau _{7}=\tau _{14}=\tau
_{27}=0$ in (\ref{v1sol}) to (\ref{v27sol}) in Theorem \ref{ThmFullDeform}.
The condition (\ref{dvcond}) is then automatically satisfied.
\end{proof}

This is equivalent to saying that $d\chi =0$ and $d\ast \chi =0$ for $\chi
=v^{e}\psi _{bcde}^{\ \ \ \ \ \ }$. This is however exactly the same
condition as the one for an infinitesimal deformation.

\begin{theorem}
\label{Thmt17tt0}Suppose $\left( \varphi ,g\right) $ is a $G_{2}$-structure
on a closed, compact manifold $M$. Consider a deformation of the $G_{2}$%
-structure $\varphi $ given by 
\begin{equation}
\varphi \longrightarrow \varphi +v^{e}\psi _{bcde}^{\ \ \ \ \ \ }
\label{lambda7deform}
\end{equation}%
If the torsion $T$ lies in the class $W_{1}\oplus W_{7}$, then this
deformation results in a torsion-free $G_{2}$-structure if and only if $T=0$
and $\nabla v=0.$
\end{theorem}

\begin{proof}
If $T=0$, from Corollary \ref{corrt0t0}, we know that the deformation (\ref%
{lambda7deform}) results in a torsion-free $G_{2}$-structure if and only if $%
\nabla v=0$. So assume now $T\neq 0$.

Now let us assume that $T\in W_{1}\subset $ $W_{1}\oplus W_{7}$ and suppose
the deformation (\ref{lambda7deform}) results in $\tilde{T}=0$. Thus here we
have $\tau _{7}=\tau _{14}=\tau _{27}=0$ and $\tilde{\tau}_{1}=\tilde{\tau}%
_{7}=\tilde{\tau}_{14}=\tilde{\tau}_{27}=0$. Then from Theorem \ref%
{ThmFullDeform}, we have 
\begin{equation}
\nabla _{a}v_{b}=\tau _{1}g_{ab}-\frac{1}{3}\tau _{1}v^{c}\varphi _{cab}
\label{delvt1t0}
\end{equation}%
and in particular, 
\begin{equation*}
dv^{\flat }=-\frac{2}{3}\tau _{1}v\lrcorner \varphi .
\end{equation*}%
Now, the consistency condition $d^{2}v^{\flat }=0$ is equivalent to either $%
\tau _{1}=0$ or 
\begin{equation*}
d\left( v\lrcorner \varphi \right) =0.
\end{equation*}%
Using (\ref{delvt1t0}) and the fact that 
\begin{equation*}
\nabla \varphi =\tau _{1}\psi ,
\end{equation*}%
we find that 
\begin{equation*}
\pi _{1}\left( d\left( v\lrcorner \varphi \right) \right) =3\tau _{1}\varphi
=0
\end{equation*}%
So we must have $\tau _{1}=0$, which gives a contradiction. Hence there are
no deformations from torsion class $W_{1}$ to $W_{0}$.

Next we assume that $T\in W_{7}\subset $ $W_{1}\oplus W_{7}$, so that only $%
\tau _{7}$ is non-vanishing. In this case, 
\begin{eqnarray}
\nabla _{a}v_{b} &=&\left( M+9\right) ^{-1}\left( -g_{ab}\left( \tau
_{7}\right) _{c}v^{c}+3\left( 1+M\right) v_{b}\left( \tau _{7}\right)
_{a}+\left( 33+M\right) v_{a}\left( \tau _{7}\right) _{b}\right.
\label{delvt7t0} \\
&&-3\varphi _{\ ab}^{c}\left( \tau _{7}\right) _{c}\left( M-3\right)
+4\left( \tau _{7}\right) _{c}v^{c}\varphi _{abd}v^{d}-24\varphi _{\ \
d[a}^{c}\left( \tau _{7}\right) _{c}v^{d}v_{b]}  \notag \\
&&\left. +12\psi _{\ \ dab}^{c}\left( \tau _{7}\right) _{c}v^{d}\right) 
\notag
\end{eqnarray}%
and correspondingly we can also get $dv^{\flat }$ from this. As before, we
consider $d\left( dv^{\flat }\right) $ and the projections of it on to $%
\Lambda _{1}^{3}$, $\Lambda _{7}^{3}$ and $\Lambda _{27}^{3}$. Let $\xi _{1}$
be the scalar corresponding to the $\Lambda _{1}^{3}$ projection, $\xi _{7}$
- the vector corresponding to the $\Lambda _{7}^{3}$ projection and $\xi
_{27}$ - the symmetric $2$-tensor corresponding to the $\Lambda _{27}^{3}$
component. As before, we can obtain scalars $\left( \xi _{7}\right)
_{a}v^{a} $ and $\left( \xi _{27}\right) _{ab}v^{a}v^{b}.$ Hence we get
three scalar equations%
\begin{eqnarray}
0 &=&16\left( M-15\right) \left( \left( \tau _{7}\right) _{a}v^{a}\right)
^{2}-6\left( 3M^{2}-34M+27\right) \left( \tau _{7}\right) ^{a}\left( \tau
_{7}\right) _{a} \\
&&-\left( M+9\right) ^{2}\nabla ^{a}\left( \tau _{7}\right) _{a}  \notag \\
0 &=&\frac{4\left( M-3\right) \left( \left( \tau _{7}\right)
_{a}v^{a}\right) ^{2}}{M+9}-\frac{9M\left( M+3\right) \left( \tau
_{7}\right) ^{a}\left( \tau _{7}\right) _{a}}{M+9}+\left( \nabla _{a}\left(
\tau _{7}\right) _{b}\right) v^{a}v^{b} \\
&&-M\nabla ^{a}\left( \tau _{7}\right) _{a}  \notag \\
0 &=&\frac{6\left( 1+M\right) \left( M-39\right) \left( \left( \tau
_{7}\right) _{a}v^{a}\right) ^{2}}{M+9}-\frac{4M\left( 5M-3\right) \left(
M-3\right) \left( \tau _{7}\right) ^{a}\left( \tau _{7}\right) _{a}}{M+9} \\
&&+2\left( M-3\right) \left( \nabla _{a}\left( \tau _{7}\right) _{b}\right)
v^{a}v^{b}-3M\left( 1+M\right) \nabla ^{a}\left( \tau _{7}\right) _{a} 
\notag
\end{eqnarray}%
We can solve these equations to get $\nabla ^{a}\left( \tau _{7}\right) _{a}$%
, $\left( \nabla _{a}\left( \tau _{7}\right) _{b}\right) v^{a}v^{b}$ and $%
\left( \tau _{7}\right) ^{a}\left( \tau _{7}\right) _{a}=\left\vert \tau
_{7}\right\vert ^{2}$ in terms of $\left( \left( \tau _{7}\right)
_{a}v^{a}\right) ^{2}=\left\langle \tau _{7},v\right\rangle ^{2}$. So in
particular, we get 
\begin{equation}
\left\vert \tau _{7}\right\vert ^{2}=\frac{3\left\langle \tau
_{7},v\right\rangle ^{2}\left( 3M^{2}-10M+51\right) }{\left(
7M^{2}-66M-9\right) M}  \label{t7t0tau7sq1}
\end{equation}%
Further, from $\xi _{7}^{d}=0$, $\varphi _{abc}v^{b}\xi _{7}^{c}=0,$ $\left(
\xi _{27}\right) _{mn}v^{n}=0$ and $\varphi _{abc}$ $\left( \xi _{27}\right)
_{\ \ n}^{b}v^{n}v^{c}=0$, we actually find that 
\begin{equation}
v=\frac{M}{\left\langle \tau _{7},v\right\rangle }\tau _{7}
\end{equation}%
and after contracting with $\tau _{7}$, we get 
\begin{equation}
\left\vert \tau _{7}\right\vert ^{2}=\frac{\left\langle \tau
_{7},v\right\rangle ^{2}}{M}  \label{t7t0tau7sq2}
\end{equation}%
Comparing (\ref{t7t0tau7sq1}) and (\ref{t7t0tau7sq2}), we get 
\begin{equation*}
\left\langle \tau _{7},v\right\rangle ^{2}\left( M+9\right) ^{2}=0
\end{equation*}%
Hence $\left\langle \tau _{7},v\right\rangle =0$ and so must have $\tau
_{7}=0$. Therefore, there are no deformations from $W_{7}$ to $W_{0}$.

Finally, suppose $T_{ab}$ lies in the strict class $W_{1}\oplus W_{7}$, so
that the $W_{1}$ component of the torsion is $\tau _{1}$ and the $W_{7}$
component is $\tau _{7}$. In this case, from Theorem \ref{ThmFullDeform}, we
have%
\begin{eqnarray}
\nabla _{a}v_{b} &=&\left( \tau _{1}-\left( \tau _{7}\right)
_{c}v^{c}\right) g_{ab}+\frac{1}{\left( M+9\right) }\left( -3\left(
M-3\right) \left( \tau _{7}\right) _{c}\varphi _{\ \ ab}^{c}\right.
\label{fulldelvt17tt0} \\
&&-\left( M+33\right) v_{a}\left( \tau _{7}\right) _{b}+3\left( 1+M\right)
\left( \tau _{7}\right) _{a}v_{b}  \notag \\
&&-\frac{1}{3}v^{c}\varphi _{cab}\left( 9\tau _{1}+\tau _{1}M-12\left( \tau
_{7}\right) _{d}v^{d}\right)  \notag \\
&&\left. +12v_{a}\varphi _{\ \ \ b}^{cd}\left( \tau _{7}\right)
_{c}v_{d}-12v_{b}\varphi _{\ \ \ a}^{cd}\left( \tau _{7}\right)
_{c}v_{d}+12\left( \tau _{7}\right) _{c}v_{d}\psi _{\ \ \ ab}^{cd}\right) 
\notag
\end{eqnarray}%
Following the general procedure outlined above, we used \emph{Maple }to
expand the necessary condition (\ref{dvcond}). Again, as before, we consider
the projections of $d\left( dv^{\flat }\right) $. As outlined above we first
consider the $\pi _{1}$, $\pi _{7}$ and $\pi _{27}$ projections of $d\left(
\left( v_{7}\right) \lrcorner \varphi +v_{14}\right) $.

Denote by $\xi _{1}$ the scalar corresponding to the $\Lambda _{1}^{3}$
component, let $\xi _{7}$ and $\xi _{27}$ be the vector and symmetric tensor
components. Then by considering the equations $\xi _{1}=0$, $\left( \xi
_{7}\right) ^{a}v_{a}=0$ and $\left( \xi _{27}\right) _{mn}v^{m}v^{n}=0$, we
can express $\left( \nabla _{a}\left( \tau _{7}\right) _{b}\right)
v^{a}v^{b} $, $\nabla ^{a}\left( \tau _{7}\right) _{a}$ and $\left\vert \tau
_{7}\right\vert ^{2}$ in terms of $M,$ $\tau _{1}$ and $\left\langle \tau
_{7},v\right\rangle $. In particular, we find that 
\begin{eqnarray}
\left\vert \tau _{7}\right\vert ^{2} &=&\frac{3\left\langle \tau
_{7},v\right\rangle ^{2}\left( 3M^{2}-10M+51\right) }{\left(
7M^{2}-66M-9\right) M}-\frac{4}{3}\frac{\tau _{1}\left\langle \tau
_{7},v\right\rangle \left( M+9\right) ^{2}}{\left( 7M^{2}-66M-9\right) }
\label{t17tt1t7sq1} \\
&&+\frac{2}{9}\frac{\tau _{1}^{2}M\left( M+9\right) ^{2}}{7M^{2}-66M-9} 
\notag
\end{eqnarray}%
Further, we can consider the vector equations $\xi _{7}^{d}=0$, $\varphi
_{abc}v^{b}\xi _{7}^{c}=0,$ $\left( \xi _{27}\right) _{mn}v^{n}=0$ and $%
\varphi _{abc}$ $\left( \xi _{27}\right) _{\ \ n}^{b}v^{n}v^{c}=0.$ From
these, in particular, we find 
\begin{equation}
\tau _{7}=\frac{\left\langle \tau _{7},v\right\rangle }{M}v  \label{t17tt1v}
\end{equation}%
So as before, we get 
\begin{equation}
\left\vert \tau _{7}\right\vert ^{2}=\frac{\left\langle \tau
_{7},v\right\rangle ^{2}}{M}  \label{t17tt1t7sq2}
\end{equation}%
Now if we equate (\ref{t17tt1t7sq1}) and (\ref{t17tt1t7sq2}), and then solve
for $\left\langle \tau _{7},v\right\rangle $, we obtain an expression for $%
\left\langle \tau _{7},v\right\rangle $ in terms of $\tau _{1}$, $\tau _{7}$
and $v$.%
\begin{equation}
\left\langle \tau _{7},v\right\rangle =\frac{M\tau _{1}}{3}
\label{t17tt1tt1sol}
\end{equation}%
Hence, 
\begin{equation}
v=\frac{3}{\tau _{1}}\tau _{7}.  \label{t17tt0tau7sol}
\end{equation}%
and, 
\begin{subequations}
\label{t17tt0MVsols}
\begin{eqnarray}
M &=&\frac{9}{\tau _{1}^{2}}\left\vert \tau _{7}\right\vert ^{2} \\
\left\langle \tau _{7},v\right\rangle &=&\frac{3}{\tau _{1}}\left\vert \tau
_{7}\right\vert ^{2}
\end{eqnarray}

Next, from equations $\left( \xi _{27}\right) _{ab}=0,$ $\varphi _{\ \ \
(a}^{cd}\left( \xi _{27}\right) _{b)d}v_{c}=0$ and $\varphi _{a}^{\ \
cd}\varphi _{b}^{\ \ \ ef}v_{c}v_{e}\left( \xi _{27}\right) _{df}=0$, we
finally obtain an expression for $\nabla _{a}\left( \tau _{7}\right) _{b}$.
Using (\ref{t17tt0tau7sol}) and (\ref{t17tt0MVsols}) to completely eliminate 
$v$ from the resulting expression, we overall get: 
\end{subequations}
\begin{equation}
\nabla \tau _{7}=\left( \frac{1}{3}\tau _{1}^{2}-\left\vert \tau
_{7}\right\vert ^{2}\right) g+5\tau _{7}\otimes \tau _{7}  \label{t17tt1dt7}
\end{equation}%
By first considering the trace of this, we find that we get the condition 
\begin{equation}
\nabla _{a}\tau _{7}^{a}+2\left( \tau _{7}\right) _{a}\left( \tau
_{7}\right) ^{a}-\frac{7}{3}\tau _{1}^{2}=0  \label{pi1condt7t1}
\end{equation}%
Recall however, that a $G_{2}$-structure in the strict torsion class $%
W_{1}\oplus W_{7}$ has 
\begin{equation*}
\tau _{7}=\nabla \left( \log \tau _{1}\right)
\end{equation*}%
So we can rewrite (\ref{pi1condt7t1}) as 
\begin{equation}
\nabla ^{2}\left( \log \tau _{1}\right) +2\left\vert \nabla \left( \log \tau
_{1}\right) \right\vert ^{2}-\frac{7}{3}\tau _{1}^{2}=0
\end{equation}%
Now note that if we let $F=\tau _{1}^{2}$, then%
\begin{equation*}
\nabla ^{2}F=\frac{14}{3}F^{2}
\end{equation*}%
Multiplying by $F$, integrating over the whole manifold $M$, and applying
Stokes's Theorem (since $M$ is closed), we get 
\begin{equation}
\int_{M}F\left( \nabla ^{2}F\right) \mathrm{vol}=-\int_{M}\left\vert \nabla
F\right\vert ^{2}\mathrm{vol}=\frac{14}{3}\int_{M}F^{3}\sqrt{\det g}\mathrm{%
vol}
\end{equation}%
However, $F=\tau _{1}^{2}$ is a positive function, so the right-hand side is
non-negative, while the left-hand side is non-positive, and we can only have
equality when both sides vanish. This happens only if $F=0$ and this implies
that both $\tau _{1}$ and $\tau _{7}$ vanish. Therefore we cannot have a
deformation from $W_{1}\oplus W_{7}$ into $W_{0}$.
\end{proof}

\begin{corollary}
Given a deformation $\varphi \longrightarrow \varphi +v^{e}\psi _{bcde}$ of
a torsion-free $G_{2}$-structure $\varphi $, the torsion of the new $G_{2}$%
-structure $\tilde{\varphi}$ will necessarily have a non-trivial component
in $W_{14}$ or $W_{27}$ unless $\nabla v=0$.
\end{corollary}

\begin{proof}
Suppose the vector $v$ defines a deformation a torsion-free $G_{2}$%
-structure results in a new $G_{2}$-structure with torsion $\tilde{T}$ lying
in $W_{1}\oplus W_{7}$, that is, there is no $W_{14}$ or $W_{27}$ component.
Then by Lemma \ref{vtildelem}, there exists a corresponding deformation,
defined by vector $\tilde{v}$, in the opposite direction from $W_{1}\oplus
W_{7}$ to a torsion-free $G_{2}$-structure. However by Theorem \ref%
{Thmt17tt0}, such a deformation exists if and only if $\tilde{\nabla}\tilde{v%
}=0$ (and equivalently, by Lemma \ref{vtildelem} $\nabla v=0$) and the
torsion $\tilde{T}$ vanishes.
\end{proof}

\begin{theorem}
There is no deformation of the form (\ref{lambda7deform}) within the strict
torsion class $W_{1}$.
\end{theorem}

\begin{proof}
We consider a $G_{2}$-structure $\left( \varphi ,g\right) $ where the only
non-zero component of torsion $T$ is $\tau _{1}$. Suppose (\ref%
{lambda7deform}) gives a deformation to a $G_{2}\,$-structure $\left( \tilde{%
\varphi},\tilde{g}\right) $ with torsion $\tilde{T}$ with the only non-zero
component being $\tilde{\tau}_{1}$. Then from Theorem \ref{ThmFullDeform},%
\begin{equation}
\nabla _{a}v_{b}=\left( \tau _{1}-\left( 1+M\right) ^{\frac{2}{3}}\tilde{\tau%
}_{1}\right) g_{ab}+4\left( 1+M\right) ^{-\frac{1}{3}}\tilde{\tau}%
_{1}v_{a}v_{b}-\frac{1}{3}v^{c}\varphi _{cab}\left( \tau _{1}-4\left(
1+M\right) ^{-\frac{1}{3}}\tilde{\tau}_{1}\right)  \label{delvt1t1}
\end{equation}%
and in particular, 
\begin{equation}
dv^{\flat }=-\frac{2}{3}v\lrcorner \varphi \left( \tau _{1}-4\left(
1+M\right) ^{-\frac{1}{3}}\tilde{\tau}_{1}\right)
\end{equation}

Then we take $d\left( dv^{\flat }\right) $, and decompose it into $\Lambda
_{1}^{3}$, $\Lambda _{7}^{3}$ and $\Lambda _{27}^{3}$ components. Since $%
d\left( dv^{\flat }\right) $ must vanish, so must each of these components.
We hence get the following equations: 
\begin{subequations}
\label{t1t1ddv}
\begin{eqnarray}
0 &=&\tau _{1}^{2}-\frac{1}{21}\frac{\left( 9M^{2}+106M+105\right) \tau _{1}%
\tilde{\tau}_{1}}{\left( 1+M\right) ^{\frac{4}{3}}}+\frac{4}{21}\frac{\left(
15M^{2}+21+28M\right) \tilde{\tau}_{1}^{2}}{\left( 1+M\right) ^{\frac{5}{3}}}
\label{t1t1pi1ddv} \\
0 &=&\left( \tau _{1}^{2}-5\frac{\tau _{1}\tilde{\tau}_{1}}{\left(
1+M\right) ^{\frac{1}{3}}}+\frac{4\tilde{\tau}_{1}^{2}}{\left( 1+M\right) ^{%
\frac{2}{3}}}\right) v^{a}  \label{t1t1pi7ddv} \\
0 &=&\left( \tau _{1}^{2}-\frac{1}{27}\frac{\left( 15M^{2}+142M+135\right)
\tau _{1}\tilde{\tau}_{1}}{\left( 1+M\right) ^{\frac{4}{3}}}+\frac{4}{27}%
\frac{\left( 21M^{2}+40M+27\right) \tilde{\tau}_{1}^{2}}{\left( 1+M\right) ^{%
\frac{5}{3}}}\right) g_{ab}  \label{t1t1pi27ddv} \\
&&+\left( \frac{8}{27}\frac{\left( 3M+5\right) \tau _{1}\tilde{\tau}_{1}}{%
\left( 1+M\right) ^{\frac{4}{3}}}-\frac{16}{27}\frac{\left( 3M+7\right) 
\tilde{\tau}_{1}^{2}}{\left( 1+M\right) ^{\frac{5}{3}}}\right) v_{a}v_{b} 
\notag
\end{eqnarray}%
Now if we contract (\ref{t1t1pi7ddv}) with $v_{a}$ and (\ref{t1t1pi27ddv})
with $v^{a}v^{b}$, we get three scalar equations: 
\end{subequations}
\begin{eqnarray}
0 &=&\tau _{1}^{2}-\frac{1}{21}\frac{\left( 9M^{2}+106M+105\right) \tau _{1}%
\tilde{\tau}_{1}}{\left( 1+M\right) ^{\frac{4}{3}}}+\frac{4}{21}\frac{\left(
15M^{2}+21+28M\right) \tilde{\tau}_{1}^{2}}{\left( 1+M\right) ^{\frac{5}{3}}}
\\
0 &=&\tau _{1}^{2}-5\frac{\tau _{1}\tilde{\tau}_{1}}{\left( 1+M\right) ^{%
\frac{1}{3}}}+\frac{4\tilde{\tau}_{1}^{2}}{\left( 1+M\right) ^{\frac{2}{3}}}
\\
0 &=&\tau _{1}^{2}+\frac{1}{9}\frac{\left( 3M^{2}-34M-45\right) \tau _{1}%
\tilde{\tau}_{1}}{\left( 1+M\right) ^{\frac{4}{3}}}+\frac{4}{9}\frac{\left(
3M^{2}+4M+9\right) \tilde{\tau}_{1}^{2}}{\left( 1+M\right) ^{\frac{5}{3}}}
\end{eqnarray}%
Here our unknowns are $\tau _{1}^{2}$, $\tau _{1}\tilde{\tau}_{1}$ and $%
\tilde{\tau}_{1}^{2}$. The determinant of the system is $\frac{32}{21}\frac{%
M^{2}}{1+M}>0,$ so the only solution is $\tau _{1}=\tilde{\tau}_{1}=0$.
\end{proof}

\section{Concluding remarks}

In this paper we have studied the deformations of $G_{2}$ structures on $7$%
-manifolds. Given a general deformation $\chi $ of a $G_{2}$-structure $%
\left( \varphi ,g\right) $, we obtained a new $G_{2}$-structure $\left( 
\tilde{\varphi},\tilde{g}\right) $, and for this new $G_{2}$-structure we
calculated its torsion tensor $\tilde{T}$ in terms of the old $G_{2}$%
-structure $\left( \varphi ,g\right) $, its torsion $T$ and the deformation $%
\chi $. We then specialized to $\chi $ lying in $\Lambda _{7}^{3}$, given by 
$\chi =-v\lrcorner \psi $. For such a deformation, in Theorem \ref{ThmvTors}%
, we computed $\tilde{T}$ in terms of $v$ and its derivatives. So then,
given a $G_{2}$-structure $\left( \varphi ,g\right) $ with particular
torsion components $\tau _{i}$, we found the equations that $v$ must satisfy
so that the the torsion of $\left( \tilde{\varphi},\tilde{g}\right) $ has
particular components $\tilde{\tau}_{i}$. For particular cases of $\tau _{i}$
and $\tilde{\tau}_{i}$, we analysed these equations. In particular, we found
that a deformation within the zero torsion class is possible if and only if $%
\nabla v=0$. In other cases, it was found that there are no deformations of
this type that take the torsion from the classes $W_{1},$ $W_{7}$ and $%
W_{1}\oplus W_{7}$ to zero. Also, it was found that there are no
deformations which take the $W_{1}$ torsion class to itself. Although these
are all mostly negative results, since deformations in $\Lambda _{7}^{3}$
are invertible, we can conclude that any such deformation with non-zero $%
\nabla v$ from a torsion-free $G_{2}$-structure will yield a $G_{2}$%
-structure whose torsion contains at least a $W_{14}$ or $W_{27}$ component.

So far we have developed a technique for computing the deformed torsion, and
of course there is a significant amount of work to be done to fully
understand deformations of other torsion classes. Deformations that lie in $%
\Lambda _{7}^{3}$ are of course the simplest possible deformations, apart
from conformal deformations, since they are defined by just a vector. Also,
as we note in Section \ref{seclam7deform}, such a deformation of a positive $%
3$-form still yield a positive $3$-form, which does define a $G_{2}$%
-structure. The ultimate aim would be to make sense of non-infinitesimal
deformations that lie in $\Lambda _{27}^{3}$. These are then defined by
traceless symmetric tensors, and moreover, not all such deformations yield
positive $3$-forms, so extra conditions need to be imposed. On the other
hand, these deformations have many more degrees of freedom than the $\Lambda
_{7}^{3}$ deformations, so we could expect to get more interesting results
and unlock many of the mysteries of $G_{2}$ manifolds. In particular, one of
the aims would be to show if there are any obstructions to deformations of $%
G_{2}$ holonomy manifolds. A even more ambitious program would be to try and
understand which $G_{2}$-structures exist on a given manifold and what is
the smallest torsion class.

\bibliographystyle{jhep-a}
\bibliography{refs2}

\end{document}